\documentclass[10pt]{article}
\usepackage{latexsym, amsmath, amssymb, amsthm}
\usepackage{enumerate, mathrsfs, bbm}
\usepackage{graphicx, tikz}
\usepackage{authblk, setspace, cite, floatrow}
\usepackage[top=1 in, bottom=1 in, left=1.2 in, right=1.2 in]{geometry}
\usepackage[pdftex,linkcolor=red,citecolor=blue,backref=page]{hyperref}

\newtheorem{theorem}{Theorem}[section]
\newtheorem{lemma}[theorem]{Lemma}
\newtheorem{corollary}[theorem]{Corollary}
\newtheorem{proposition}[theorem]{Proposition}

\newtheorem{definition}[theorem]{Definition}

\newtheorem{conjecture}[theorem]{Conjecture}

\numberwithin{equation}{section}

\newfloatcommand{capbtabbox}{table}[][\FBwidth]

\makeatletter
\def\blfootnote{\xdef\@thefnmark{}\@footnotetext}
\makeatother

\def\m{\mathbb}  \def\mcal{\mathcal}  \def\mb{\mathbbm}		
\def\eps{\epsilon}  \def\a{\alpha}    \def\lam{\lambda}	\def\O{\Omega}
\def\p{\partial}  \def\ls{\lesssim}	\def\gs{\gtrsim}
\def\la{\langle}  \def\ra{\rangle}  \def\wh{\widehat}	\def\wt{\widetilde}	
\def\i{\int\limits}
\def\be{\begin{equation}}     \def\ee{\end{equation}}
\def\bp{\begin{pmatrix}}	\def\ep{\end{pmatrix}}

\begin{document}
	\title{Well-posedness and Critical Index Set  of the Cauchy Problem for the Coupled KdV-KdV Systems on $\mathbb{T}$}
	
	\author{Xin Yang and Bing-Yu Zhang}
	
	\date{}
	\maketitle
	
	\begin{abstract}
		Studied in this paper is   the well-posedness of the Cauchy problem for the coupled KdV-KdV systems 
		\begin{equation}
			\label{z-1}
			\left\{\begin{array}{rcl}
				u_t+a_{1}u_{xxx} &=& c_{11}uu_x+c_{12}vv_x+d_{11}u_{x}v+d_{12}uv_{x},\vspace{0.03in}\\
				v_t+a_{2}v_{xxx}&=& c_{21}uu_x+c_{22}vv_x +d_{21}u_{x}v+d_{22}uv_{x}, \vspace{0.03in}\\
				\left. (u,v)\right |_{t=0} &=& (u_{0},v_{0})
			\end{array}\right.\end{equation}
		posed on the periodic domain $\mathbb{T}$ in the  following four spaces
		\[\mbox{$  {\cal H}^s_1:=H^s_0 (\mathbb{T})\times H^s_0 (\mathbb{T}), \;\; {\cal H}^s_2:=H^s_0 (\m{T})\times H^s(\mathbb{T}), \;\; {\cal H}^s_3:=H^s (\mathbb{T})\times H^s_0 (\mathbb{T}), \;\; {\cal H}^s_4:=H^s (\mathbb{T})\times H^s (\mathbb{T}).$}\] 
		The coefficients are assumed to satisfy $a_1 a_2\neq 0$ and $\sum\limits_{i,j}(c_{ij}^2+d_{ij}^2)>0$.
		
		Fix $k\in\{1,2,3,4\}$.  Then for any coefficients $a_1$, $a_2$,  $(c_{ij})$ and $(d_{ij})$,  it is shown that there exists a critical index $s^*_k \in (-\infty, +\infty]$ such that the  system (\ref{z-1}) is analytically locally well-posed in $\mathcal{H}^s_k$ if $s>s^*_k$ but weakly analytically ill-posed if $s<s^{*}_k$.  Viewing $s^*_k$ as a function of the coefficients, its range $\mcal{C}_k$ is defined to be  the {\em critical index set} for the analytical well-posedness of  (\ref{z-1}) in  $\mcal{H}^s_k$.  
		
		By investigating some properties of the {\em irrationality exponents} of the real numbers and by establishing some sharp bilinear estimates in non-divergence form,   we manage to identify $ {\cal C}_1=  \left \{ -\frac12, \infty \right\} \bigcup \big[\frac12, 1\big]$ and ${\cal C}_q= \left \{ -\frac12, -\frac14, \infty \right\} \bigcup \big[\frac12, 1\big]$ for $q=2,3,4$.
		In particular,  these sets contain an open interval $(\frac12,1)$. This is in sharp contrast to the $\m{R}$ case  in which the critical index set ${\cal C}$  for the analytical well-posedness of (\ref{z-1}) in the space $H^s (\m{R})\times H^s (\m{R})$ consists of exactly four numbers: $ {\cal C}=\left  \{ -\frac{13}{12}, -\frac34, 0, \frac34 \right \}.$
		
	\end{abstract}
	
	
	
	
	
	\section{Introduction}\label{Sec, intro}
	\subsection{cKdV systems and well-posedness}
	This paper studies the Cauchy problem for the coupled KdV-KdV (cKdV) systems 
	posed  on the periodic domain $\m{T}=\m{R}/(2\pi\m{Z})$  in the following general form:
	\be\label{ckdv, general}
	\left\{\begin{array}{ll}
		\bp u_t\\v_t  \ep + A_{1}\bp u_{xxx}\\v _{xxx}\ep = A_{2}\bp uu_x\\vv_x \ep + A_{3}\bp u_{x}v\\uv_{x}\ep, & 
		\vspace{0.1in}\\
		\left.  \bp u\\v\ep\right |_{t=0} =\bp u_{0}\\v_{0}\ep,
	\end{array}\right.\ee
	where $\{A_{i}\}_{1\leq i\leq 3}$ are $2\times 2$ real  constant matrices,  $u=u(x,t)$ and  $ 
	v =v(x,t) $ are real-valued functions of the variables  $x\in\m{T}$ and $t\in\m{R}$.  It is assumed  that  there exists an invertible real matrix $M$ such that {\small $A_{1}=M\bp a_{1}& 0\\0& a_{2}\ep M^{-1}$} with $a_{1}a_2 \ne 0$.  By regarding {\small $M^{-1}\bp u\\v\ep$} as the new unknown functions (still denoted by $u$ and $v$),   it  reduces  to
	\be\label{ckdv, vector form}
	\left\{\begin{array}{ll}
		\bp u_t\\v_t  \ep + \bp a_{1} & 0 \\ 0 & a_{2}\ep\bp u_{xxx}\\v _{xxx}\ep = C\bp uu_x\\vv_x \ep + D\bp u_{x}v\\uv_{x}\ep, & 
		\vspace{0.1in}\\
		\left.  \bp u\\v\ep\right |_{t=0} =\bp u_{0}\\v_{0}\ep,
	\end{array}\right.\ee
	or equivalently,
	\be\label{ckdv, coef form}
	\left\{\begin{array}{rcll}
		u_t+a_{1}u_{xxx} &=& c_{11}uu_x+c_{12}vv_x+d_{11}u_{x}v+d_{12}uv_{x}, & 
		\vspace{0.03in}\\
		v_t+a_{2}v_{xxx} &=& c_{21}uu_x+c_{22}vv_x +d_{21}u_{x}v+d_{22}uv_{x}, & 
		\vspace{0.03in}\\
		\left. (u,v)\right |_{t=0} &=& (u_{0},v_{0}),
	\end{array}\right. \ee
	which is called in {\em divergence form}  if $d_{11}=d_{12}$ and $d_{21}=d_{22}$,  and in {\em non-divergence form} otherwise.  
	
	Listed below  are  a few specializations of (\ref{ckdv, general}) appeared in the literature.
	\begin{itemize}
		\item Majda-Biello system: 
		\be\label{M-B system}
		\left\{\begin{array}{rll}
			u_{t}+u_{xxx} &=& -vv_{x}, \\
			v_{t}+a_{2}v_{xxx} &=& -(uv)_{x}, \\
			\left. (u,v)\right |_{t=0} &=& (u_{0},v_{0}),
		\end{array}\right.\ee
		where $a_{2}\neq 0$. This system was proposed by Majda and Biello in \cite{MB03}. 
		
		\item Hirota-Satsuma system: 
		\be\label{H-S system}
		\left\{\begin{array}{rll}
			u_{t}+ a_{1}u_{xxx} &=& -6a_{1}uu_{x}+c_{12}vv_{x},   \vspace{0.03in}\\
			v_{t}+v_{xxx} &=& -3uv_{x}, \vspace{0.03in}\\
			\left. (u,v)\right |_{t=0} &=& (u_{0},v_{0}),
		\end{array}\right.\ee
		where $a_{1}\neq 0$. This system is in non-divergence form and was proposed by Hirota-Satsuma in \cite{HS81}.
		
		\item Gear-Grimshaw system: 
		\be\label{G-G system}
		\left\{\begin{array}{rcl}
			u_{t}+u_{xxx}+\sigma_{3}v_{xxx} &=&-uu_{x}+\sigma_{1}vv_{x}+\sigma_{2}(uv)_{x},\vspace{0.03in}\\
			\rho_{1}v_{t}+\rho_{2}\sigma_{3}u_{xxx}+v_{xxx}+\sigma_{4}v_{x} &=& \rho_{2}\sigma_{2}uu_{x}-vv_{x}+\rho_{2}\sigma_{1}(uv)_{x}, \vspace{0.03in}\\
			\left. (u,v)\right |_{t=0} &=& (u_{0},v_{0}),
		\end{array}\right.\ee
		where $\sigma_{i}\in\m{R}(1\leq i\leq 4)$ and $\rho_{1},\,\rho_{2}>0$. When $\sigma_4=0$, this system is a special case of (\ref{ckdv, general}) with
		\be\label{A1 in G-G system}
		A_{1}=\bp 1& \sigma_{3}\\ \frac{\rho_{2}\sigma_{3}}{\rho_{1}} & \frac{1}{\rho_{1}} \ep.\ee
		Note that $A_{1}$ in (\ref{A1 in G-G system}) is diagonalizable over $\m{R}$ for any $\sigma_{3}\in\m{R}$ and $\rho_{1},\,\rho_{2}>0$. Moreover, the eigenvalues of $A_{1}$ are nonzero unless $\rho_{2}\sigma_{3}^{2}=1$. So (\ref{G-G system}) can be reduced to the form (\ref{ckdv, coef form}) as long as $\rho_{2}\sigma_{3}^{2}\neq 1$. This system was derived by Gear-Grimshaw in \cite{GG84} (also see \cite{BPST92} for the physical explanation).
	\end{itemize}

	We are mainly concerned with the well-posedness of the  Cauchy problem for the system (\ref{ckdv, coef form}) in some  scaled Banach  spaces $ {\cal X}^s (\m{T}) $.
	
	\begin{definition}\label{Def, WP}
		For any integer $l\geq 0$,  the Cauchy problem of the system (\ref{ckdv, coef form})  is said to be $C^{l}$-locally well-posed (LWP) in the space  $ {\cal X}^s (\m{T}) $  if for any $\delta  > 0$, there exists $T>0$  such that
		\begin{itemize}
			\item[(a)]  for any $(u_0, v_0) \in  {\cal X}^s (\m{T}) $  with $\|(u_0, v_0) \|_{{\cal X}^s (\m{T}) }\leq \delta $, (\ref{ckdv, coef form}) admits a unique solution $(u,v)$ in
			the space $C([0,T]; {\cal X}^s (\m{T}))$  satisfying the auxiliary condition $ w\in \m{Y}^T_s$, where  $\m{Y}^T_s$  is an auxiliary metric space;
			\item[(b)] the corresponding  solution  map $K$ from the initial data  $(u_0, v_0)$ to the solution $ (u, v) $: $K(u_0,v_0)=  (u,v)$,  is $C^l$ smooth from $\{ (u_0, v_0):\  \| (u_0, v_0) \|_{{\cal X}^s(\m{T})} \leq \delta \}$  to $C([0, T];{\cal X}^s (\m{T}))\cap \m{Y}^T_s$.
		\end{itemize}
		Similarly, if the solution map $K$ is  real analytic   (or   uniformly continuous), then the system (\ref{ckdv, coef form})  is said to be analytically (or uniformly continuously) LWP in ${\cal X}^s(\m{T})$.
	\end{definition}
	
	In the literature, the well-posedness in Definition \ref{Def, WP} is called {\it conditional}  (cf. \cite{Kat1995, BSZ2004} and the references therein). If the uniqueness holds in the space $C([0,T]; {\cal X}^s (\m{T}))$ without the auxiliary space $\m{Y}^T_s$, then the well-posedness is said to be {\it unconditional}.  On the other hand,  the well-posedness in Definition \ref{Def, WP} is called {\it local}, but if the lifespan $T$ can be specified independently of $\delta $, then the system (\ref{ckdv, coef form}) is said to be globally well-posed (GWP). On the other hand, if the local well-posedness fails, then it is called ill-posed.
	
	\begin{definition}\label{Def, IP}
		For any integer $l\geq 0$, the Cauchy problem of the systems (\ref{ckdv, coef form}) is said to be $C^{l}$ ill-posed (IP) in the space $ {\cal X}^s (\m{T}) $  if there exists $\delta>0$ such that for any $T>0$, there does not exist a $C^{l}$ smooth solution map $K$ as defined in Definition \ref{Def, WP}. Similarly, if there does not exist a real analytic (or uniformly continuous) solution map, then (\ref{ckdv, coef form}) is said to be analytically (or uniformly continuously) IP in $ {\cal X}^s (\m{T}) $.
	\end{definition}
	
	In the rest of the paper, for the convenience of notations, the analytically LWP, GWP and IP will be written as A-LWP, A-GWP and A-IP respectively. Similarly, the uniformly continuously LWP, GWP and IP will be written as U-LWP, U-GWP and U-IP respectively.
	
	\subsection{Literature review}
	For  the single KdV equation
	\be\label{KdV eq}
	u_t+u_{xxx}+ uu_x  = 0, \qquad u(x,0)= u_0 (x)
	\ee 
	posed on either $\m{R}$ or $\m{T}$, the study of the its well-posedness began in the late 1960s with the work of Sj\"oberg  \cite{Sjob-1,Sjob-2}  and has culminated in the work of Killip and Visan\cite{KV19}.  Many significant breakthroughs were made in this study of more than fifty years (see e.g.  Bona-Smith \cite{BS75}, Kato \cite{Kat1975, Kat1979,Kat1981, Kat1983},  Constatin-Saut \cite{CS88}, Kenig-Ponce-Vega \cite{KPV89,  KPV91Indiana,KPV91JAMS,KPV93CPAM,KPV96}, Bourgain \cite{Bou93b},  Christ-Colliander-Tao \cite{CCT03}, Colliander-Keel-Staffilani-Takaoka-Tao \cite{CKSTT03, CKSTT04},  Kappeler -Topalov \cite{KT06},  Molinet \cite{Mol11, Mol12},  Tao \cite{Tao01},  Killip-Visan \cite{KV19}, to name a few).
	
	Define $H^{s}(\m{R})$ and $H^{s}(\m{T})$ as the standard Sobolev spaces on $\m{R}$ and $\m{T}$ respectively. In addition, denote $H_{0}^{s}(\m{T})$ to be the collection of the functions in $H^{s}(\m{T})$ whose mean value is 0:
	\[H_{0}^{s}(\m{T})=\Big\{f\in H^{s}(\m{T}): \int_{0}^{2\pi}f(x)\,dx=0\Big\}.\]   
	The following theorem summarizes the known results on the well-posedness of (\ref{KdV eq}).
	
	\begin{theorem}[\cite{KT06,KV19,Mol11,Mol12,KPV96,Guo09,Kis09,
			CKSTT03,CCT03,Bou93b}]
		\label{Thm, WP for KdV}\quad
		The Cauchy   problem (\ref{KdV eq}) is 
		\begin{itemize}
			\item[(1)] $C^0$-GWP in both spaces $H^s (\m{R})$  and $H^s (\m{T})$ for any $s\geq -1$, and $C^0$-IP for any $s<-1$; 
			\item[(2)] A-GWP in the space $H^s (\m{R})$ for any $s\geq -\frac34$, and U-IP for any $s< -\frac34$; 
			\item[(3)] A-GWP in the space $H^s_0 (\m{T})$ for any $s\geq -\frac12$, and U-IP for any $s< -\frac12$;
			\item[(4)] U-IP in the space $H^s (\m{T})$ for any $s\in \m{R}$.
		\end{itemize}
	\end{theorem}
	
	For the cKdV systems (\ref{ckdv, coef form}) posed on $\m{R}$, it has also been very well studied (see e.g. \cite{AC08, BPST92, Oh09, YZ1812}  and the references therein). Of particular interest is a special class of (\ref{ckdv, coef form}) in the following form
	\be\label{weak nonlin ckdv}
	\left\{\begin{array}{rcl}
		u_{t}+a_{1}u_{xxx}&=& d_1(uv)_{x},\vspace{0.03in}\\
		v_{t}+a_{2}v_{xxx}&=& d_2(uv)_{x}, \vspace{0.03in}\\
		\left. (u,v)\right |_{t=0} &=& (u_{0},v_{0}).
	\end{array}\right.\ee
	It was shown in \cite{YZ1812} that it is A-LWP in the space  $H^s(\m{R}) \times H^s(\m{R})$ for any $s\geq -\frac{13}{12}$ if $a_{1}a_{2}<0$.
	This  is rather  surprising  since even the single KdV equation (\ref{KdV eq}) is $C^0$-IP in $H^s(\m{R})$ for any $s<-1$ (see \cite{Mol11}). Similar phenomenon was also found in the coupled mKdV systems \cite{CP19}.
	
	For the cKdV (\ref{ckdv, coef form}) posed on $\m{T}$ in divergence form, its A-LWP has been thoroughly investigated by Oh \cite{Oh09}. As a specific application, \cite{Oh09} obtained sharp A-LWP results for the Majda-Biello system (\ref{M-B system}) for any $a_2\neq 0$. The most interesting case is when $a_2\in(0,4]\setminus\{1\}$. In this case, (\ref{M-B system}) is A-LWP in $H_{0}^{s}(\m{T})\times H^{s} (\m{T})$ for any $s> \min\{1,\,\tilde{s}(a_2)\}$, where $\tilde{s}(a_2)\geq \frac12$ is some constant only depending on $a_2$. 
	On the other hand, (\ref{M-B system}) is $C^{3}$-IP if $s<\min\{1,\,\tilde{s}(a_2)\}$. So the regularity threshold for the A-LWP of (\ref{M-B system}) in the space $H_{0}^{s}(\m{T})\times H^{s} (\m{T})$ is $\min\{1,\tilde{s}(a_2)\}$.
	
	By contrast, the study for the cKdV (\ref{ckdv, coef form}) posed on $\m{T}$ in non-divergence form is far from complete. Angulo \cite{Ang04,Ang05} studied the Hirota-Satsuma system (\ref{H-S system}) and obtained the following results. When $a_{1}\neq 0,1$, (\ref{H-S system}) is A-LWP in $H^{s}_0(\m{T})\times H^{s}_0(\m{T})$ for any $s\geq 1$ and when $a_1=1$, (\ref{H-S system}) is A-LWP in $L^{2}_0(\m{T})\times H^{1}_0(\m{T})$. 
	\footnote{Actually, the mean-zero condition on $v$ is not needed, although it is needed on $u$. In fact, the mean-zero condition on $v$ is not applicable since the Hirota-Satsuma system (\ref{H-S system}) does not preserve the mean value of $v$. As a result, it makes more sense to adjust the spaces $H^{s}_0(\m{T})\times H^{s}_0(\m{T})$ and $L^{2}_0(\m{T})\times H^{1}_0(\m{T})$ in \cite{Ang04, Ang05} to be $H^{s}_0(\m{T})\times H^{s}(\m{T})$ and $L^{2}_0(\m{T})\times H^{1}(\m{T})$ respectively.}
	Comparing to the known results on the cKdV systems in divergence form, there still exist a lot of room to explore and to optimize the regularity thresholds for the non-divergence cases.


	\subsection{Bilinear estimates}
	In Theorem \ref{Thm, WP for KdV}, the optimal regularity index -1 for either the $\m{R}$ case or the $\m{T}$ case was obtained by the inverse scattering method \cite{KT06, KV19}. This method relies on the complete integrability of the single KdV equation and usually leads to $C^0$ well-posedness. Another method, called the bilinear estimate approach, does not require the complete integrability property and usually gives the sharp analytical well-posedness. For example, by using this method, \cite{KPV96} discovered the best regularity index for the A-LWP in the $\m{R}$ case (and the $\m{T}$ case resp.) to be $-\frac34$ (and $-\frac12$ resp.). 
	
	Since most cKdV systems do not possess the complete integrability property, the bilinear estimate approach seems to be the most powerful tool to use. This method is based on the 
	Fourier restriction spaces $X^{\a}_{s,b}$ which was first introduced by Bourgain in \cite{Bou93b}.
	
	\begin{definition}[\cite{Bou93b, KPV96}]\label{Def, FR space-classical}
		Let $G=\m{R}$ (or $\m{T}$) and denote $G^{*}=\m{R}$ (or $\m{Z}$) to be the dual group of $G$. For any $\a,\,s,\,b\in\m{R}$ with $\a\neq 0$, the Fourier restriction space $X^{\a}_{s,b}(G\times\m{R})$ is defined to be the completion of the Schwartz space $\mathscr{S}(G\times\m{R})$ with respect to the norm
		\be\label{FR norm-classical}
		\| w\|_{X^{\a}_{s,b}(G\times\m{R})}:=\|\la k\ra ^{s}\la\tau-\a k^3\ra ^{b}\wh{w}(k,\tau)\|_{L^{2}(G^{*}\times\m{R})},\ee
		where $\la\cdot\ra:=1+|\cdot|$ and $\wh{w}$ refers to the space-time Fourier transform of $w$.
	\end{definition}
	
	Then the following is the so-called bilinear estimate for the KdV equation (\ref{KdV eq}).
	
	\be\label{class bilin}
	\|\p_{x}(uv)\|_{X^{1}_{s,b-1}}\ls \| u\|_{X^{1}_{s,b}}\| v\|_{X^{1}_{s,b}}, \quad\forall\, u,v\in X^{1}_{s,b},\ee
	where in the $\m{T}$ case, $u$ and $v$ are also required to have zero means, i.e. $u(\cdot,t),\,v(\cdot,t)\in H^{s}_{0}(\m{T})$ for any $t\in\m{R}$.
	
	Bourgain \cite{Bou93b} first proved (\ref{class bilin}) for $s=0$ and $b=\frac12$ in both $\m{R}$ and $\m{T}$  cases, which lead to the A-LWP of (\ref{KdV eq}) in $H^{s}(\m{R})$ or in $H^{s}_0(\m{T})$ for any $s\geq 0$. Later, Kenig, Ponce and Vega\cite{KPV96} optimized this estimate.
	\begin{lemma}[\cite{KPV96}]
		\label{Lemma, kpv}\quad
		\begin{itemize}
			\item In the $\m{R}$ case, (\ref{class bilin}) holds for any $s>-\frac34$ with some $b=b(s)>\frac12$, but fails for any $s<-\frac34$ and $b\in\m{R}$.
			
			\item In the $\m{T}$ case, (\ref{class bilin}) holds for any $s\geq -\frac12$ with $b=\frac12$, but fails if $s<-\frac12$ or $b\neq \frac12$.
		\end{itemize}
	\end{lemma}
	
	For the cKdV systems (\ref{ckdv, coef form}), the situations are more complicated since four types of bilinear estimates, (\ref{div 1})--(\ref{nondiv 2}), need to be studied.
	
	Divergence forms:
	\begin{eqnarray}
		&\text{(D1):} &\quad \|\p_{x}(w_{1}w_{2})\|_{X^{\a_2}_{s,b-1}}\ls \|w_{1}\|_{X^{\a_1}_{s,b}}\|w_{2}\|_{X^{\a_1}_{s,b}}, \label{div 1} \\
		&\text{(D2):} &\quad \|\p_{x}(w_{1}w_{2})\|_{X^{\a_1}_{s,b-1}}\ls \|w_{1}\|_{X^{\a_1}_{s,b}}\|w_{2}\|_{X^{\a_2}_{s,b}}, \label{div 2}
	\end{eqnarray}
	
	Non-divergence forms:
	\begin{eqnarray}
		&\text{(ND1):} &\quad
		\|(\p_{x}w_{1})w_{2}\|_{X^{\a_1}_{s,b-1}}\ls \| w_{1}\|_{X^{\a_1}_{s,b}}\|w_{2}\|_{X^{\a_2}_{s,b}}, \label{nondiv 1}\\
		&\text{(ND2):} &\quad
		\|w_{1}(\p_{x}w_{2})\|_{X^{\a_1}_{s,b-1}}\ls \| w_{1}\|_{X^{\a_1}_{s,b}}\|w_{2}\|_{X^{\a_2}_{s,b}}. \label{nondiv 2}
	\end{eqnarray}
	where $w_1$ (or $w_2$) refers to $u$ or $v$, and $\a_1$ (or $\a_2$) represents $a_1$ or $a_2$ in (\ref{ckdv, coef form}). 
	\footnote{
		Strictly speaking, in the $\m{T}$ case, the above bilinear estimates need to be considered in more complicated forms and some zero-mean conditions should be imposed on $w_1$ and $w_2$. But in order to illustrate the idea more clearly, we drop these technical details and just take the forms which are consistent with the $\m{R}$ case. For precise bilinear estimates in the $\m{T}$ case, please see Corollary \ref{Cor, bilin est, div}, Theorem \ref{Thm, bilin est, non-div} and Section \ref{Sec, bilin est, non-div}.
	}
	There are various ways to write out these estimates, what we adopted here is to fix the first term on the right hand side to be $\|w_1\|_{X^{\a_1}_{s,b}}$. 
	
	In the $\m{R}$ case, all the sharp regularity indices have been found for the bilinear estimates (\ref{div 1})--(\ref{nondiv 2}). 
	
	\begin{lemma}[\cite{AC08, Oh09, YZ1812}]
		\label{Lemma, sharp bilin-R}
		Let $\a_1,\a_2\in\m{R}\setminus\{0\}$ and denote $r=\frac{\a_2}{\a_1}$. Then for any bilinear estimate among (\ref{div 1})--(\ref{nondiv 2}), there exists a critical index $s^{*}$, as shown in Table \ref{Table, bilin est on R}, such that this bilinear estimate holds for any $s>s^{*}$ with some $b=b(s)>\frac12$, but fails for any $s<s^{*}$ and $b\in\m{R}$.
		
		\begin{table}[!ht]
			\renewcommand\arraystretch{1.8}
			\begin{center}
				\begin{tabular}{|c|c|c|c|c|c|} \hline
					$s^{*}$   & $r<0$    &  $0<r<\frac{1}{4}$ & $r=\frac{1}{4}$  & $r>\frac{1}{4}$, $r\neq 1$ & $r=1$  \\ \hline
					(D1) & $ -\frac{3}{4}$ &  $ -\frac{3}{4}$ & $ \frac{3}{4}$  &  $ 0$ & $ -\frac{3}{4}$  \\\hline
					(D2) & $ -\frac{13}{12}$  &  $ -\frac{3}{4}$  & $ \frac{3}{4}$ &  $ 0$ & $ -\frac{3}{4}$       \\\hline
					(ND1) &  $ -\frac{3}{4}$ &  $ -\frac{3}{4}$ & $  \frac{3}{4}$ & $0$ &  $ 0$   \\\hline
					(ND2) &  $ -\frac{3}{4}$  & $ -\frac{3}{4}$  & $ \frac{3}{4}$ &  $ 0$  & $ 0$   \\\hline
				\end{tabular}	 
			\end{center} 
			\caption{Sharp bilinear estimates on $\m{R}$ ($r=\frac{\a_2}{\a_1}$)}
			\label{Table, bilin est on R}
		\end{table}
	\end{lemma}
	
	From the above table, we can see there are four critical indices $\big\{-\frac{13}{12},\,-\frac{3}{4},\,0,\,\frac34\big\}$ for the bilinear estimates associated to the cKdV systems posed on $\m{R}$.
	
	In the $\m{T}$ case, Oh\cite{Oh09} studied the Majda-Biello system\cite{MB03} whose associated bilinear estimates are (\ref{M-B, bilin, d1}) and (\ref{M-B, bilin, d2}).
	\begin{eqnarray}
		\|\p_{x}(v^2)\|_{X^{1}_{s,b-1}} &\ls & \| v\|_{X^{a_2}_{s,b}}\| v\|_{X^{a_2}_{s,b}}, \label{M-B, bilin, d1}\\
		\|\p_{x}(vu)\|_{X^{a_2}_{s,b-1}} &\ls & \| v\|_{X^{a_2}_{s,b}} \| u\|_{X^{1}_{s,b}}. \label{M-B, bilin, d2}
	\end{eqnarray}
	Note that the above two estimates are special (essentially equivalent) cases of (\ref{div 1}) and (\ref{div 2}). In fact,
	\begin{itemize}
		\item (\ref{M-B, bilin, d1}) corresponds to (\ref{div 1}) by setting $w_1=w_2=v$, $\a_1=a_2$ and $\a_2=1$.
		
		\item (\ref{M-B, bilin, d2}) corresponds to (\ref{div 2}) by setting $w_1=v$, $w_2=u$, $\a_1=a_2$, and $\a_2=1$.
	\end{itemize}
	In both cases, the ratio between $\a_2$ and $\a_1$ is $r=\frac{1}{a_2}$. The most interesting case is when $a_2\in(0,4]\setminus\{1\}$, i.e. when $r\in\big[\frac14,\infty\big)\setminus\{1\}$. In this case, the key technical issue is to know how close the rational numbers can approximate a given real number. Oh\cite{Oh09} adopted the so-called {\it minimal type index} $\nu(\rho)$ for any real number $\rho$ (see Definition \ref{Def, MTI}) to capture this approximation. Let $c_1,c_2$ be the roots of the quadratic equation $3a_2 x^2-3a_2 x+a_2-1=0$ and let $d_1,d_2$ be the roots of the quadratic equation $(1-a_2)x^2+3a_2 x-3a_2=0$. More specifically, let
	\be\label{c and d in Oh-paper}\begin{split}
		c_1=\frac12-\frac{1}{6}\sqrt{\frac{12}{a_2}-3}, &\quad c_2=\frac12+\frac{1}{6}\sqrt{\frac{12}{a_2}-3}, \\
		&\\
		d_{1}=\frac{-3a_2-\sqrt{3a_2(4-a_2)}}{2(1-a_2)}, &\quad d_2=\frac{-3a_2+\sqrt{3a_2(4-a_2)}}{2(1-a_2)}.
	\end{split}\ee
	Denote 
	\be\label{Oh index}
	\nu_c=\max\{\nu(c_1),\nu(c_2)\}, \quad \nu_{d}=\max\{\nu(d_1),\,\nu(d_2)\}.\ee
	\begin{lemma}[\cite{Oh09}]
		\label{Lemma, sharp bilin-T by Oh}
		Let $a_2\in (0,4]\setminus\{1\}$. Define $\nu_c$ and $\nu_d$ as in (\ref{Oh index}). Then for the bilinear estimate (\ref{M-B, bilin, d1}) or (\ref{M-B, bilin, d2}), there exists a critical index $s^{*}$, as shown in Table \ref{Table, Oh-bilin est on T}, such that this bilinear estimate holds for any $s>s^{*}$ with $b=\frac12$, but fails for any $s<s^{*}$ and $b\in\m{R}$.
		\begin{table}[!ht]
			\renewcommand\arraystretch{1.8}
			\begin{center}
				\begin{tabular}{|c|c|} \hline
					$s^{*}$  & $a_2\in (0,4]\setminus\{1\}$  \\ \hline
					(\ref{M-B, bilin, d1}) &  $ \min\{1,\frac12+\frac12\nu_c\}$   \\\hline
					(\ref{M-B, bilin, d2}) &  $ \min\{1,\frac12+\frac12\nu_d\}$       \\\hline
				\end{tabular}	 
			\end{center} 
			\caption{Sharp bilinear estimates on $\m{T}$ by Oh \cite{Oh09}}
			\label{Table, Oh-bilin est on T}
		\end{table}
	\end{lemma}
	
	Based on this discovery, Oh further concludes the sharp A-LWP for the Majda-Biello system (\ref{M-B system}).
	\begin{theorem}[\cite{Oh09}]
		\label{Thm, Oh-M-B}
		Let $a_2\in (0,4]\setminus\{1\}$. Define 
		\[s^{*}_{1}=\min\Big\{1,\, \frac12+\frac12\max\{\nu_{c},\nu_{d}\} \Big\}.\]
		Then (\ref{M-B system}) is A-LWP for any $s>s^{*}_{1}$, but $C^3$-IP for any $s<s^{*}_{1}$, in the space $H^{s}_{0}(\m{T})\times H^{s}(\m{T})$.
	\end{theorem}
	
	\subsection{Weakly analytical ill-posedness}
	For the single KdV equation (\ref{KdV eq}), it follows from Theorem \ref{Thm, WP for KdV} and Lemma \ref{Lemma, kpv} that the critical indices for the bilinear estimate (\ref{class bilin}) match those for the A-LWP of (\ref{KdV eq}) in both $\m{R}$ and $\m{T}$ cases. For Majda-Biello system (\ref{M-B system}), we can also see from Lemma \ref{Lemma, sharp bilin-T by Oh} and Theorem \ref{Thm, Oh-M-B} that (\ref{M-B system}) is A-LWP if and only if both the corresponding bilinear estimates (\ref{M-B, bilin, d1}) and (\ref{M-B, bilin, d2}) hold. 
	
	Inspired by these observations, it is conjectured that for any cKdV system (\ref{ckdv, coef form}), its regularity threshold for the A-LWP is equivalent to that for the associating bilinear estimates. In fact, by the standard argument in \cite{Bou93b,KPV96,CKSTT03}, once the associated bilinear estimates are justified for some index $s$, then the A-LWP can also be established for the same $s$. In other words, if a cKdV system is A-IP, then at least one of the corresponding bilinear estimates must fail. But whether the failure of the bilinear estimates implies the A-IP is not known in general. This motivates the following definition.
	
	\begin{definition}\label{Def, weak IP}
		The Cauchy problem of the single KdV equation (\ref{KdV eq}) is said to be weakly A-IP if the bilinear estimate (\ref{class bilin}) fails. Similarly, the cKdV systems (\ref{ckdv, coef form}) are said to be weakly A-IP if at least one of the corresponding bilinear estimates fails.
	\end{definition}
	
	Based on the above definition, the aforementioned conjecture is translated to the equivalence between A-IP and weakly A-IP. 
	
	\begin{conjecture}\label{Conj, A-IP same as W-IP}
		For any cKdV system (\ref{ckdv, coef form}), it is A-IP if and only if it is weakly A-IP.
	\end{conjecture}
	
	Since A-IP always implies weakly A-IP, the conjecture further reduces to ``weakly A-IP impies A-IP". For the single KdV and the Majda-Biello system, this conjecture has been confirmed. But for some special cKdV, say (\ref{weak nonlin ckdv}) posed on $\m{R}$, although the critical index for the corresponding bilinear estimate has been found to be $-\frac{13}{12}$ in Lemma \ref{Lemma, sharp bilin-R}, it is still unknown if it is A-IP in $H^{s}(\m{R})\times H^{s}(\m{R})$ for any $s<-\frac{13}{12}$.

	\subsection{Critical index set}
	Consider the cKdV systems (\ref{ckdv, coef form}) posed on $\m{R}$. Denote $\mcal{H}^{s}(\m{R})=H^{s}(\m{R})\times H^{s}(\m{R})$.
	
	\begin{definition}\label{Def, cis for ckdv on R}
		For any cKdV system (\ref{ckdv, coef form}) posed on $\m{R}$ with $a_1a_2\neq 0$ and either $C=(c_{ij})$ or $D=(d_{ij})$ does not vanish. 
		\begin{itemize}
			\item If there exists $s^{*}\in\m{R}$, depending on $a_1$, $a_2$, $C$ and $D$, such that (\ref{ckdv, coef form}) is A-LWP in the space $\mcal{H}^{s}(\m{R})$ for any $s>s^{*}$, but (weakly) A-IP for any $s<s^{*}$ in $\mcal{H}^{s}(\m{R})$, then this $s^{*}$ is called the (weakly) analytically critical index. If there does not exist such a critical index, then we define $s^{*}=\infty$.
			
			\item Regarding the above critical index $s^{*}$ as a map: $s^{*}={\cal G} (a_1,a_2, C, D)$, then the range of this map, denoted as $\mcal{C}$, is called the (weakly) analytically critical index set of (\ref{ckdv, coef form}) in $\mcal{H}^{s}(\m{R})$.
		\end{itemize}
	\end{definition}
	
	In the following, we will write "analytically critical index`` to be "A-critical index``. Based on Definition \ref{Def, cis for ckdv on R}, it follows from Lemma \ref{Lemma, sharp bilin-R} that the weakly A-critical index set of (\ref{ckdv, coef form}) in $\mcal{H}^{s}(\m{R})$ is 
	\[\mbox{$\mcal{C}=\big\{-\frac{13}{12},\,-\frac34,\,0,\,\frac34\big\}.$}\]
	Now we consider the following (\ref{xkdv}) which is a variant of the single KdV equation. 
	\begin{equation}\label{xkdv}
		u_t + a u_{xxx} = c uu_x, \qquad u(x,0)=u_0(x), \qquad ac\ne 0.  
	\end{equation}
	Similar to Definition \ref{Def, cis for ckdv on R}, we can define the A-critical index set of (\ref{xkdv}) posed on $\m{R}$ or $\m{T}$. For any $l\geq 0$, we can also define the $C^{l}$-critical index set analogously. By some simple scaling computation, one can see that the values of $a$ and $c$ in (\ref{xkdv}) do not affect its well-posedness. So it follows from Theorem \ref{Thm, WP for KdV} that 
	\begin{itemize}   
		\item The $C^{0}$-critical index set of (\ref{xkdv}) in $H^s (\m{R})$ or $H^s (\m{T})$ is $\{ -1\}$.
		\item The A-critical index set of (\ref{xkdv}) is $\big\{ -\frac34\big\}$ in the space $H^s (\m{R})$, $\big\{-\frac12 \big\}$ in the space $H^s_0 (\m{T})$, and $\{ \infty \}$ in the space $H^s (\m{T})$.
	\end{itemize}
	
	Two observations can be drawn from the above results. Firstly, unlike (\ref{xkdv}), the coefficients of the cKdV systems (\ref{ckdv, coef form}) do have an effect on the A-critical index. Secondly, if the problem is posed on $\m{T}$, then the critical index can be different if the underlying spaces are different (e.g. $H^{s}_{0}(\m{T})$ cf. $H^{s}(\m{T})$). 
	
	The main purpose of this paper is to investigate the weakly A-critical index set of the cKdV systems (\ref{ckdv, coef form}) posed on $\m{T}$. We will consider the following four spaces:
	\be\label{spaces}
	\mcal{H}^{s}_{1}:=H^{s}_{0}(\m{T})\times H^{s}_{0}(\m{T}),\;\; \mcal{H}^{s}_{2}:= H^{s}_{0}(\m{T})\times H^{s}(\m{T}),\;\; \mcal{H}^{s}_{3}:= H^{s}(\m{T})\times H^{s}_{0}(\m{T}),\;\; \mcal{H}^{s}_{4}:= H^{s}(\m{T})\times H^{s}(\m{T}).\ee
	
	\begin{definition}\label{Def, cis for ckdv on T}
		Let $k\in\{1,2,3,4\}$. We define $\mcal{C}_{k}$ to be the weakly A-critical index set of the cKdV systems (\ref{ckdv, coef form}) in the space $\mcal{H}^{s}_{k}$.
	\end{definition}
	
	From the results in Oh\cite{Oh09} on the Majda-Biello system (\ref{M-B system}), the weakly A-critical index $s^{*}(a_2)$ in the space $\mcal{H}^{s}_{2}$ is 
	\[s^{*}(a_2)=\left\{\begin{array}{cll}
		-\frac12 &\text{if} & a_2\in(-\infty,0)\cup (4,\infty),\\
		\infty &\text{if} & a_2=1, \\
		\min\{1,\tilde{s}(a_2)\} &\text{if} & a_2\in(0,4]\backslash\{1\}.
	\end{array}\right.\]
	where $\tilde{s}(a_2)\geq \frac12$ is some constant only depending on $a_2$. So the weakly A-critical index set for (\ref{M-B system}) in the space $\mcal{H}^{s}_2$ is 
	\[\mbox{$\big\{-\frac12,\infty\big\}\bigcup \Big\{\min\{1,\tilde{s}(a_2)\}:a_2\in(0,4]\setminus\{1\}\Big\}.$}\] 
	But what is the precise range of $\tilde{s}(a_2)$ over the region $(0,4]\backslash\{1\}$? This question was not answered in \cite{Oh09}, we will find an answer in this paper.
	
	\subsection{Irrationality exponent}
	In order to find out the critical index set for (\ref{ckdv, coef form}) posed on $\m{T}$, as we mentioned before, it is crucial to estimate how well the real number can be approximated by rational functions. In the number theory, such estimate is called the Diophantine approximation. 
	\footnote{The interested readers are referred to S. Lang \cite{Lan95} and Y.  Bugeaud \cite{Bug04}  for a nice introduction to this subject.
	}
	One standard characterization is via the irrationality exponent.
	
	\begin{definition}[see e.g. \cite{BBS16}]\label{Def, IE}
		A real number $\rho$  is said to be   approximable with power $\mu $ if the inequality
		\be\label{approx power}
		0<\Big| \rho-\frac{m}{n}\Big|<\frac{1}{|n|^{\mu}}\ee
		holds for infinitely many $(m,n)\in\m{Z}\times\m{Z}^*$,
		and 
		\be\label{E}
		\mu(\rho):=\sup\{\mu \in\m{R}: \text{$\rho$ is approximable with power $\mu$}\}\ee
		is called the irrationality exponent  of $\rho$.
	\end{definition}	
	
	
	For any rational number $\rho$,  $\mu(\rho)=1$. For any irrational number  $\rho$, $\mu(\rho)\geq 2$, but exact value of $\mu (\rho)$ is difficult to find  in general. The basic properties of $\mu(\rho)$ are collected in Proposition \ref{Prop, IE}. Recalling the minimal type index $\nu(\rho)$ in Oh's paper \cite{Oh09}, it is closely related to $\mu(\rho)$ but is defined slightly different.
	
	\begin{definition}[see e.g. \cite{Arn88, Oh09}]\label{Def, MTI}
		A real number $\rho$ is said to be  of type $\nu$ if there exists a positive constant $K=K(\rho,\nu)$ such that  the inequality
		\be\label{type nu}
		\Big|\rho-\frac{m}{n}\Big|\geq \frac{K}{|n|^{2+\nu}}\ee
		holds for any $(m,n)\in\m{Z}\times\m{Z}^{*}$, and 
		\be\label{MTI}
		\nu(\rho):=\inf\{\nu\in\m{R}:\text{$\rho$ is of type $\nu$}\}\ee
		is called the  minimal type index  of $\rho$, 
		where the infimum is understood as $\infty$ if the set $\{\nu\in\m{R}:\text{$\rho$ is of type $\nu$}\}$ is empty.
	\end{definition}
	
	If $\rho\in\m{Q}$, then it is easy to see $\nu(\rho)=\infty$ and $\mu(\rho)=1$. If $\rho\in\m{R}\setminus\m{Q}$, then it will be shown in Proposition \ref{Prop, IE and MTI relation} that $\nu(\rho)=\mu(\rho)-2$. 
	
	\subsection{Main results}
	Recall the irrationality exponent $\mu(\cdot)$ in Definition \ref{Def, IE}. For $r\geq \frac{1}{4}$,  define 
	\be\label{sharp index by IE}
	\sigma_{r}=\mu(\sqrt{12r-3})\quad\text{and}\quad s_{r}=\left\{\begin{array}{cll}
		1 & \text{if} & \sigma_{r}=1 \text{\; or \;} \sigma_{r}\geq 3,\\
		\dfrac{\sigma_{r}-1}{2} & \text{if} & 2\leq \sigma_{r}<3.
	\end{array}\right.\ee
	The basic properties of $s_{r}$ are collected in Proposition \ref{Prop, s_r}.  
	
	For the Majda-Biello system (\ref{M-B system}) with $a_2\in(0,4]\setminus\{1\}$,  Oh\cite{Oh09} showed the critical regularity indexes for the associated bilinear estimates to be $\min\big\{1,\frac12+\frac12\nu_{c}\big\}$ and $\min\big\{1,\frac12+\frac12\nu_{d}\big\}$, see Table \ref{Table, Oh-bilin est on T}. The first main result of this paper is to demonstrate these two indexes are actually the same.
	
	\begin{theorem}\label{Thm, same index in Oh-paper}
		Let $a_2\in(0,4]\setminus\{1\}$ and define $\nu_{c}$ and $\nu_{d}$ as in (\ref{Oh index}). Then $\nu_{c}=\nu_{d}$ and 
		\be\label{simple index}
		\min\Big\{1,\frac12+\frac12 \nu_{c}\Big\}=s_{\frac{1}{a_2}},
		\ee
		where $s_{\frac{1}{a_2}}$ is defined as in (\ref{sharp index by IE}).
	\end{theorem} 
	
	The key observation in the proof of this theorem is the invariance of the irrationality exponent under the reciprocal operation, that is $\mu(\rho)=\mu(1/\rho)$, see Proposition \ref{Prop, inv of IE}. 
	
	Combining Theorem \ref{Thm, same index in Oh-paper} with Lemma \ref{Lemma, sharp bilin-T by Oh} (also see Oh\cite{Oh09}), we are able to write out the sharp regularity indexes for the bilinear estimates (\ref{div 1}) and (\ref{div 2}) in a more unified way.
	
	\begin{corollary}\label{Cor, bilin est, div}
		Let $\a_1,\a_2\in\m{R}\setminus\{0\}$ and denote $r=\frac{\a_2}{\a_1}$. Define $s_{r}$ as in (\ref{sharp index by IE}) for $r\geq \frac14$. Then for the bilinear estimate (\ref{div 1}) or (\ref{div 2}), there exists a critical index $s^{*}$, as shown in Table \ref{Table, bilin est on T, div}, such that this bilinear estimate holds for any $s>s^{*}$ with $b=\frac12$, but fails for any $s<s^{*}$ and $b\in\m{R}$.
		
		\begin{table}[!ht]
			\renewcommand\arraystretch{1.2}
			\begin{center}
				\begin{tabular}{|c|c|c|c|} \hline
					$s^{*}$  & $r<\frac14, r\neq 0$ & $r\geq \frac14$, $r\neq 1$ & $r=1$ ($\wh{w_1}(0,\cdot)=\wh{w_2}(0,\cdot)=0$) \\ \hline
					(D1) & $ -\frac12$  &  $\min\{1,s_r\}$ &  $ -\frac12$    \\\hline
					(D2) with $\wh{w_2}(0,\cdot)=0$ & $ -\frac12$ &  $\min\{1,s_r\}$  & $ -\frac12$   \\\hline
				\end{tabular}	 
			\end{center} 
			\caption{Sharp bilinear estimates on $\m{T}$ in divergence form ($r=\frac{\a_2}{\a_1}$)}
			\label{Table, bilin est on T, div}
		\end{table}
	\end{corollary}
	Let $\mcal{U}$  be the range of $s_{r}$: ${ \cal U}=\left  \{ s_{r}:  \frac14\leq r < \infty \right \}$. Then it will be shown in Proposition \ref{Prop, s_r} that $\mcal{U}=\big[\frac12,1\big]$. As a result, we obtain the following conclusion (Recall that $H^{s}_{2}:=H^{s}_{0}(\m{T})\times H^{s}(\m{T})$ in (\ref{spaces})).
	
	\begin{corollary}
		For the Majda-Biello systems (\ref{M-B system}) posed on $\m{T}$ with $a_2\neq 0$, its weakly A-critical index set in the space $H^{s}_{2}$ is 
		$\big\{-\frac12,\infty\big\} \bigcup \big[\frac12,1\big]$.
	\end{corollary}
	
	Next, we study the cKdV systems (\ref{ckdv, coef form}) in non-divergence form. The main result is the following sharp bilinear estimates. 
	
	\begin{theorem}\label{Thm, bilin est, non-div}
		Let $\a_1,\a_2\in\m{R}\setminus\{0\}$ and denote $r=\frac{\a_2}{\a_1}$. Define $s_{r}$ as in (\ref{sharp index by IE}) for $r\geq \frac14$. Then for the bilinear estimate (\ref{nondiv 1}) or (\ref{nondiv 2}), there exists a critical index $s^{*}$, as shown in Table \ref{Table, bilin est on T, nondiv}, such that this bilinear estimate holds for any $s>s^{*}$ with $b=\frac12$, but fails for any $s<s^{*}$ and $b\in\m{R}$.
		
		\begin{table}[!ht]
			\renewcommand\arraystretch{1.2}
			\begin{center}
				\begin{tabular}{|c|c|c|c|} \hline
					$s^{*}$  & $r<\frac14, r\neq 0$ & $r\geq \frac14$, $r\neq 1$ & $r=1$ \\ \hline
					(ND1) ($\wh{w_2}(0,\cdot)=0$) & $ -\frac14$  &  $\min\{1,s_r\}$ &  $ \frac12$    \\\hline
					(ND2)& $ -\frac14$ &  $\min\{1,s_r\}$  & $ \frac12$ ($\wh{w_1}(0,\cdot)=0$) \\\hline
				\end{tabular}	 
			\end{center} 
			\caption{Sharp bilinear estimates on $\m{T}$ in non-divergence form ($r=\frac{\a_2}{\a_1}$)}
			\label{Table, bilin est on T, nondiv}
		\end{table}
	\end{theorem}
	
	As a corollary, consider the Hirota-Satusma systems (\ref{H-S system}). 
	\begin{corollary}
		For the Hirota-Satsuma systems (\ref{H-S system}) posed on $\m{T}$ with $a_1\neq 0$ and $c_{12}\in\m{R}$, its weakly A-critical index set in the space $H^{s}_{2}$ is 
		$\big\{-\frac14,\infty\big\} \bigcup\big[\frac12,1\big]$.
	\end{corollary}
	The more detailed well-posedness results on the Hirota-Satsuma system are shown in the appendix.
	Based on the results in Table \ref{Table, bilin est on T, div} and \ref{Table, bilin est on T, nondiv}, we can actually study the well-posedness of cKdV systems (\ref{ckdv, coef form}) in the space $\mcal{H}^{s}_{k}$ for any $k\in\{1,2,3,4\}$. The following is a summary of their weakly A-critical index sets.

	\begin{theorem}\label{Thm, critical index set}  
		The weakly A-critical index sets $\mcal{C}_{k}\,(k=1,2,3,4)$ in Definition \ref{Def, cis for ckdv on T}  are
		\be\label{critical index sets} 
		\mbox{${\cal C}_1= \big \{ -\frac12, \infty  \big \} \bigcup \big[\frac12,1\big] \quad\text{and}\quad {\cal C}_q =\big  \{ -\frac12, -\frac14, \infty \big  \} \bigcup  \big[\frac12,1\big], \ q=2,3,4 .$ }\ee
	\end{theorem} 
	
	Fix any $k\in \{ 1,2,3,4\}$. The above result enables us to provide a complete classification of the cKdV systems (\ref{ckdv, coef form}) in $H^{s}_{k}$.
	
	\begin{theorem}[Classification of the systems (\ref{ckdv, coef form})]  \label{Thm, classify}   
		Assume $a_1a_2 \ne 0$ and either $C$ or $D$ does not vanish. Fix $k\in \{ 1,2,3,4\} $ and define $\mcal{H}^s_k$ as in (\ref{spaces}).   Then the systems (\ref{ckdv, coef form}) are completely classified into a family of  classes, each of which corresponds to a unique index $s^{*}\in {\cal C}_k$  such that any system in this class is A-LWP in the space  $ {\cal H}^s_k$  if $s>s^*$ while it is weakly A-IP if $s<s^{*}$.
	\end{theorem}

	\subsection{Organization of the paper}
	The organization of the rest of the paper is as follows. 
	In Section \ref{Sec, FR space}, we will present the definitions of the Fourier restriction spaces and the resonance functions. 
	Section \ref{Sec, proof of Thm on IE} is devoted to explore properties of the irrationality exponents and prove Theorem \ref{Thm, same index in Oh-paper}.  
	Some linear estimates will be introduced in Section \ref{Sec, lin est}. 
	Theorem \ref{Thm, bilin est, non-div} will be broken into Lemma \ref{Lemma, bilin est, nondiv 1}, and \ref{Lemma, bilin est, nondiv 2} and Proposition \ref{Prop, sharp bilin est, nondiv 1} and \ref{Prop, sharp bilin est, nondiv 2} in Section \ref{Sec, bilin est, non-div}.  
	We will justify Lemma \ref{Lemma, bilin est, nondiv 1} and \ref{Lemma, bilin est, nondiv 2} in Section \ref{Sec, proof of bilin} and prove Proposition \ref{Prop, sharp bilin est, nondiv 1} and \ref{Prop, sharp bilin est, nondiv 2} in Section \ref{Sec, proof of sharp bilin}. 
	Finally, Appendix \ref{Appendix,  H-S} includes the analytical well-posedness results about the Hirota-Satsuma systems (\ref{H-S system}) which can be obtained as a corollary of Theorem \ref{Thm, bilin est, non-div}.

	\section{Fourier restriction spaces on $\m{T}$}
	\label{Sec, FR space}
	
	To study the LWP of (\ref{ckdv, coef form}), 
	we adopt the similar treatment  as in \cite{CKSTT03}  to deal with more general periodic problem posed  on $\m{T}_{\lam}=\m{R}/(2\pi\lam\m{Z})$  for $\lam\geq 1$ and  thus consider the system 
	\be\label{ckdv on T_lam, vector}
	\left\{\begin{array}{ll}
		\bp u_t\\v_t  \ep + \bp a_{1} & 0 \\ 0 & a_{2}\ep\bp u_{xxx}\\v _{xxx}\ep = C\bp uu_x\\vv_x \ep + D\bp u_{x}v\\uv_{x}\ep, & x\in\m{T_{\lam}},\, t\in\m{R}, \vspace{0.1in}\\
		\left.  \bp u\\v\ep\right |_{t=0} =\bp u_{0}\\v_{0}\ep\in H^{s}(\m{T}_{\lam})\times H^{s}(\m{T}_{\lam}).
	\end{array}\right.\ee
	We use $\mcal{F}_x$, $\mcal{F}_t$ and $\mcal{F}$ to denote the spatial, temporal and space-time Fourier transform respectively. However, when there is no confusion, we simply use $\,\hat{}\,$ to denote any of these three types of Fourier transforms. On the other hand, $\mcal{F}_x^{-1}$, $\mcal{F}_t^{-1}$ and $\mcal{F}^{-1}$ represent the corresponding inverse Fourier transforms.
	
	The temporal Fourier transform and its inverse are defined standardly. The definitions of the spatial Fourier transform and its inverse are more complicated. Denote the frequency space corresponding to $\m{T}_{\lam}$  by 
	$\m{Z}_{\lam}=\big\{k \,\big|\, k=n/\lam\,\,\, \text{for some $n\in\m{Z}$}\big\}$. The normalized counting measure $dk^{\lam}$ on $\m{Z}_{\lam}$ is defined by 
	\be\label{measure}
	\int_{\m{Z}_{\lam}}a(k)dk^{\lam}=\frac{1}{\lam}\sum_{k\in\m{Z}_{\lam}}a(k).\ee
	For any function $f(x)$ on $\m{T}_{\lam}$, 
	\be\label{FT, spatial}
	(\mcal{F}_{x}f)(k):=\int_{0}^{2\pi\lam}e^{-ikx}f(x)\,dx, \quad\forall\, k\in \m{Z}_{\lam}.\ee
	On the other hand, for any functon $g(k)$ on $\m{Z}_{\lam}$, 
	\be\label{IFT, spatial}
	(\mcal{F}_{x}^{-1}g)(x):=\frac{1}{2\pi}\int_{\m{Z}_{\lam}}e^{ixk}g(k)\,dk^{\lam}, \quad\forall\, x\in\m{R}.\ee
	Consequently, the norm for the Sobolev space $H^{s}(\m{T}_{\lam})$  is defined as 
	\be\label{Sobolev norm}
	||f||_{H^{s}(\m{T}_{\lam})}=||\la k\ra^{s}\hat{f}(k)||_{L^{2}(\m{Z}_{\lam})}=\Big(\frac{1}{\lam}\sum_{k\in\m{Z}_{\lam}}\la k\ra^{2s}|\hat{f}(k)|^{2}\Big)^{\frac12},\ee
	where $\la\cdot\ra:=1+|\cdot|$. The homogeneous subspace $H_{0}^{s}(\m{T}_{\lam})$ is defined as $H_{0}^{s}(\m{T}_{\lam})=\{f\in H^{s}(\m{T}_{\lam}): \hat{f}(0)=0\}$. Finally, the space-time Fourier transform and its inverse are defined as $\mcal{F}=\mcal{F}_{t}\mcal{F}_x$ and $\mcal{F}^{-1}=\mcal{F}_{x}^{-1}\mcal{F}_{t}^{-1}$.
	
	For any $\alpha\neq 0$ and $\lam\geq 1$, consider 
	\be\label{linear eq}
	\left\{\begin{array}{ll}
		\p_{t}w+\alpha\p_{x}^{3}w=0, & x\in\m{T}_{\lam},\, t\in\m{R},\\
		w(0)=w_{0}\in H^{s}(\m{T}_{\lam}).
	\end{array}\right.\ee
	The solution to (\ref{linear eq}) is given explicitly by 
	\be\label{semigroup op}
	w(x,t)=\int_{\m{Z}_{\lam}}e^{ikx}e^{i\phi^{\alpha}(k)t}\widehat{w_{0}}(k)dk^{\lam} := S^{\a}_{\lam}(t)w_{0},\ee
	with  
	\be\label{phase fn}
	\phi^{\alpha}(k):=\alpha k^{3}.\ee
	The following is a generalized version of (\ref{Def, FR space-classical}) for the definition of the Fourier restriction spaces on $\m{T}$.
	
	\begin{definition}[\cite{Bou93b, KPV96, CKSTT03}]\label{Def, FR space}
		For any $\a,\,s,\,b,\,\lam\in\m{R}$ with $\a\neq 0$ and $\lam\geq 1$, the Fourier restriction space $X^{\a}_{s,b,\lam}$ is defined to be the completion of the Schwartz space $\mathscr{S}(\m{T}_{\lam}\times\m{R})$ with respect to the norm
		\be\label{FR norm}
		\| w\|_{X^{\a}_{s,b,\lam}}:=\|\la k\ra ^{s}\la\tau-\phi^{\a}(k)\ra ^{b}\wh{w}(k,\tau)\|_{L^{2}(\m{Z}_{\lam}\times\m{R})},\ee
		where $\wh{w}$ refers to the space-time Fourier transform of $w$. 
	\end{definition}
	
	It has been pointed out in \cite{KPV96} that one needs to take $b=\frac{1}{2}$ for the periodic case. However, this space barely fails to be in $C\big(\m{R}_{t}; H^{s}_{x}\big)$. To ensure the continuity of the time flow of the solution,  a  smaller space $Y^{\a}_{s,\lam}$  will be used via the norm 
	\be\label{Y space}
	\| w\|_{Y^{\a}_{s,\lam}}:=\| w\|_{X^{\a}_{s,\frac{1}{2},\lam}}+\|\la k\ra^{s}\wh{w}(k,\tau)\|_{L^{2}(\m{Z}_{\lam};L^{1}(\m{R}))}.\ee
	Since the second term $\|\la k\ra^{s}\wh{w}(k,\tau)\|_{L^{2}(\m{Z}_{\lam};L^{1}(\m{R}))}$ has already dominated the $L^{\infty}_{t}H^{s}_{x}$ norm of $w$, it follows that $Y^{\a}_{s,\lam}\subseteq C\big(\m{R}_{t}; H^{s}_{x}\big)$. The companion spaces $Z^{\a}_{s,\lam}$ via the norm (\ref{Z space})  is then  introduced to control the $Y^{\a}_{s,\lam}$ norm of the integral term from the Duhamel principle (see Lemma \ref{Lemma, lin est for KdV}).
	\be\label{Z space}
	\| w\|_{Z^{\a}_{s,\lam}}=\| w\|_{X^{\a}_{s,-\frac{1}{2},\lam}}+\left\|\frac{\la k\ra^{s}\wh{w}(k,\tau)}{\la\tau-\phi^{\a}(k)\ra}\right\|_{L^{2}(\m{Z}_{\lam};L^{1}(\m{R}))}.\ee
	
	For convenience, we will drop $\lam$ when it equals $1$. That is,  
	$X^{\a}_{s,b}:=X^{\a}_{s,b,1}$, $ Y^{\a}_{s}:=Y^{\a}_{s,1}$ and $ Z^{\a}_{s}:=Z^{\a}_{s,1}$.
	Throughout this paper, 
	\[\m{R}^*:=\m{R}\setminus\{0\},\quad \m{Z}^*:=\m{Z}\setminus\{0\},\quad \m{Z}_{\lam}^*:=\m{Z}_{\lam}\setminus\{0\}, \quad\m{Q}^*:=\m{Q}\setminus\{0\}.\]
	
	\begin{definition}[\cite{Tao01}]\label{Def, res fcn}
		Let $(\a_1,\a_2,\a_3)$ be a triple in $(\m{R}^{*})^{3}$. Define the {\it resonance function} $H$ associated to this triple by
		\be\label{res fcn}
		H(k_1,k_2,k_3)=\sum_{i=1}^{3}\phi^{\a_i}(k_i),\quad \forall\,\sum_{i=1}^{3}k_{i}=0,\ee
		where $\phi^{\a_i}(k_i)$ is as defined in (\ref{phase fn}). The {\it resonance set} of $H$ is defined to be the zero set of $H$, that is
		\[\Big\{(k_1,k_2,k_3)\in\m{R}^3:\sum_{i=1}^{3}k_i=0,\,H(k_1,k_2,k_3)=0\Big\}.\] 
		
	\end{definition}
	
	\section{Proof of Theorem \ref{Thm, same index in Oh-paper}}
	\label{Sec, proof of Thm on IE}
	
	We first collect some classical results about the irrationality exponent $\mu(\rho)$.
	\begin{proposition}\label{Prop, IE}
		\quad
		\begin{enumerate}[(a)]
			\item If $\rho\in\m{Q}$, then $\mu(\rho)=1$;
			\item If $\rho\in\m{R}\setminus\m{Q}$, then $\mu(\rho)\geq 2$;
			\item If $\rho$ is an irrational algebraic number, then $\mu(\rho)=2$; 
			\item For almost every $\rho\in\m{R}$, $\mu(\rho)=2$;
			\item The function $\mu$ maps $\m{R}$ onto $\{1\}\cup [2,\infty]$.
		\end{enumerate}
	\end{proposition}
	\begin{proof}
		Part (a) and (b) are standard. Part (c) is the famous Thue-Siegel-Roth theorem \cite{Thu1909, Sie1921, Rot55}. Part (d) is the Khintchine theorem \cite{Khi1924}. Part (e) was proved by Jarnik\cite{Jar1929, Jar1931} using the theory of continued fractions. 
	\end{proof}
	
	Next, we present two technical   lemmas  concerning the irrationality exponent function $\mu$.
	
	\begin{lemma}\label{Lemma, seq for IE}
		If a real number $\rho$ is approximable with power $\mu$, then there exist $\{(m_{j},n_{j})\}_{j=1}^{\infty}\subset \m{Z}\times\m{Z}^*$ with the following two properties.
		\begin{enumerate}[(1)]
			\item $\{n_j\}_{j=1}^{\infty}$ is an increasing positive sequence and $\lim\limits_{j\rightarrow\infty}n_j=\infty$;
			
			\item For each $j\geq 1$, $(m_j,n_j)$ satisfies 
			\[0<\Big| \rho-\frac{m_j}{n_j}\Big|<\frac{1}{n_j^{\mu}}.\]
		\end{enumerate}
	\end{lemma}
	
	\begin{proof}
		The conclusion follows from Definition \ref{Def, IE} and the observation that for any fixed $n\in\m{Z}^*$, there are at most finitely many $m\in\m{Z}$ such that 
		\[0<\Big| \rho-\frac{m}{n}\Big|<\frac{1}{|n|^{\mu}}.\]
	\end{proof}
	
	\begin{lemma}\label{Lemma, bdd approx beyond IE}
		Let $\rho\in\m{R}\setminus\m{Q}$ with $\mu(\rho)<\infty$. Then for any $\eps>0$, there exists a constant $K=K(\rho,\eps)>0$ such that the inequality 
		\be\label{bad approx beyond IE}
		\Big|\rho-\frac{m}{n}\Big|\geq \frac{K}{|n|^{\mu_{\eps}}}\ee
		holds for any $(m,n)\in\m{Z}\times\m{Z}^*$, where $\mu_{\eps}=\mu(\rho)+\eps$.
	\end{lemma}
	
	\begin{proof}
		Since $\mu_{\eps}=\mu(\rho)+\eps>\mu(\rho)$,  it follows from Definition \ref{Def, IE} that there exist at most finitely many $(m,n)\in\m{Z}\times\m{Z}^*$ such that 
		\[0<\Big|\rho-\frac{m}{n}\Big|<\frac{1}{|n|^{\mu_\eps}}.\]
		In addition, since $\rho$ is an irrational number, $\big|\rho-\frac{m}{n}\big|$ is never zero. Therefore,
		by choosing a sufficiently small constant $K=K(\rho,\eps)$, (\ref{bad approx beyond IE}) holds for any  $(m,n)\in\m{Z}\times\m{Z}^*$.
	\end{proof}
	
	Now some invariant properties of the irrationality exponent $\mu$ will be justified.
	
	\begin{proposition}\label{Prop, inv of IE} 
		\quad
		\begin{enumerate}[(a)]
			\item For any $\sigma\in \m{Q}^{*}$ and $\rho\in\m{R}$, $\mu(\sigma\rho)=\mu(\rho)$.
			\item For any $\sigma\in \m{Q}$ and $\rho\in\m{R}$, $\mu(\rho+\sigma)=\mu(\rho)$.
			\item For any $\rho\in\m{R}^{*}$, $\mu\big(\frac{1}{\rho}\big)=\mu(\rho)$.
		\end{enumerate}
	\end{proposition}
	\begin{proof}
		As part (a) and (b) are obvious,  we will only prove part (c).  First, by taking advantage of (a), we may assume $\rho>0$. Then due to symmetry, it reduces to  prove 
		\[\mbox{$\mu(\rho)\leq \mu\big(\frac{1}{\rho}\big)$.}\]
		If $\rho\in\m{Q}$, then it is trivial. So we further assume $\rho$ is an irrational number, which implies  $\mu\big(\frac{1}{\rho}\big)\geq 2$. Then it suffices to show that if $\rho$ is approximable with some power $\mu\geq 2$, then $\frac{1}{\rho}$ is approximable with the power $\mu-\eps$ for any $\eps>0$. 
		
		Let $\rho$ be approximable with some power $\mu\geq 2$ and fix any $\eps>0$. It follows from Lemma \ref{Lemma, seq for IE} that 
		there exists a sequence $\{(m_{j},n_{j})\}_{j=1}^{\infty}\subset \m{Z}\times\m{Z}^*$ such that $n_j>0$, $\lim\limits_{j\rightarrow\infty}n_j=\infty$ and 
		\[0<\Big|\rho-\frac{m_j}{n_j}\Big|<\frac{1}{n_j^{\mu}}.\]
		Since $\rho>0$ and $\mu\geq 2$, it immediately yields that $\lim\limits_{j\rightarrow\infty}m_j=\infty$. In addition, for sufficiently large $j$, we have $0<\frac{m_j}{n_j}<2\rho$. Hence, for any such $j$,  noticing
		\begin{align*}
			0<\Big| \frac{1}{\rho}-\frac{n_j}{m_j}\Big| = \frac{n_j}{\rho m_j}\Big|\frac{m_j}{n_j}-\rho\Big| <\frac{1}{\rho m_{j} n_{j}^{\mu-1}} &=\frac{1}{\rho}\Big(\frac{m_j}{n_j}\Big)^{\mu-1}\frac{1}{m_j^{\mu}}.
		\end{align*}
		we obtain 
		\[0<\Big| \frac{1}{\rho}-\frac{n_j}{m_j}\Big|<\frac{(2\rho)^{\mu-1}}{\rho}\frac{1}{m_j^{\mu}}.\]
		Finally, since $\lim\limits_{j\rightarrow\infty}m_j=\infty$, when $j$ is large enough, $m_{j}^{\eps}>(2\rho)^{\mu-1}/\rho$. Therefore, 
		\[0<\Big| \frac{1}{\rho}-\frac{n_j}{m_j}\Big|<\frac{1}{m_j^{\mu-\eps}}.\]
		Since the above inequality is valid for any large $j$, $\frac{1}{\rho}$ is approximable with the power $\mu-\eps$.
	\end{proof}

	Finally, we discuss the relation between the irrationality exponent $\mu(\rho)$ in Definition \ref{Def, IE} and the minimal type index $\nu(\rho)$ in Definition \ref{Def, MTI}. 
	
	\begin{proposition}\label{Prop, IE and MTI relation}
		\quad
		\begin{enumerate}[(a)]
			\item If $\rho\in\m{Q}$, then $\nu(\rho)=\infty$ and $\mu(\rho)=1$;
			\item If $\rho\in\m{R}\setminus\m{Q}$, then $\nu(\rho)=\mu(\rho)-2$.
		\end{enumerate}
	\end{proposition}
	\begin{proof}
		Part (a) is obvious, so we will only prove part (b). Let $\rho\in\m{R}\setminus\m{Q}$. We will first show $\nu(\rho)\geq \mu(\rho)-2$. If $\rho$ is approximable with power $\mu$, then it follows from Lemma \ref{Lemma, seq for IE} that there exists a sequence $\{(m_{j},n_{j})\}_{j=1}^{\infty}\subset \m{Z}\times\m{Z}^*$ such that $n_j>0$, $\lim\limits_{j\rightarrow\infty}n_j=\infty$ and 
		\[\Big|\rho-\frac{m_j}{n_j}\Big|<\frac{1}{n_j^{\mu}}.\]
		As a result, for any $\eps>0$, there does not exist $K=K(\rho,\eps)$ such that 
		\[\Big|\rho-\frac{m_j}{n_j}\Big|\geq\frac{K}{n_j^{\mu-\eps}}\]
		holds for all $j\geq 1$. In other words, $\rho$ is not of type $\mu-2-\eps$ according to Definition \ref{Def, MTI}. Therefore, $\nu(\rho)\geq \mu-2-\eps$. Sending $\eps\rightarrow 0^+$ and then taking supremum with respect to $\mu$ leads to $\nu(\rho)\geq \mu(\rho)-2$.
		
		Now if $\mu(\rho)=\infty$, then it follows from $\nu(\rho)\geq \mu(\rho)-2$ that $\nu(\rho)=\infty$. So in the following, we just assume $\mu(\rho)<\infty$ and intend to show $\nu(\rho)\leq \mu(\rho)-2$. According to Lemma \ref{Lemma, bdd approx beyond IE}, for any $\eps>0$, there exists a constant $K=K(\rho,\eps)>0$ such that  
		\[\Big|\rho-\frac{m}{n}\Big|\geq \frac{K}{|n|^{\mu(\rho)+\eps}}\]
		holds for any $(m,n)\in\m{Z}\times\m{Z}^*$. Hence, $\rho$ is of type $\mu(\rho)+\eps-2$, which implies $\nu(\rho)\leq \mu(\rho)+\eps-2$. Sending $\eps\rightarrow 0^+$ yields $\nu(\rho)\leq \mu(\rho)-2$.
	\end{proof}
	
	For $r\geq \frac14$, denote $\sigma_r$ and $s_r$ as in (\ref{sharp index by IE}). That is,
	\be\label{sigma_r and s_r} 
	\sigma_r=\mu(\sqrt{12r-3})\quad\text{and}\quad s_{r}=\left\{\begin{array}{cll}
		1 & \text{if} & \sigma_{r}=1 \text{\; or \;} \sigma_{r}\geq 3,\\
		\dfrac{\sigma_{r}-1}{2} & \text{if} & 2\leq \sigma_{r}<3.
	\end{array}\right.\ee
	The properties of $s_r$ can be derived from Proposition \ref{Prop, IE}.
	
	\begin{proposition}\label{Prop, s_r}
		Let $r\geq \frac{1}{4}$  be given. Then 
		\begin{enumerate}[(a)]
			\item $s_{r}=1$ if   $\sqrt{12r-3}\in\m{Q}$;
			
			\item $\frac12\leq s_{r}\leq 1$ for any $r\geq \frac{1}{4}$;
			\item $s_{r}=\frac12$ if $r$ is an algebraic number and  $\sqrt{12r-3}\notin\m{Q}$.
			
			\item $s_{r}=\frac12$ for almost every $r\geq \frac14$;
			
			\item The range of $s_{r}$ over $r\in[\frac14,\infty)$ is $[\frac12,1]$.
		\end{enumerate}
	\end{proposition}
	\begin{proof}
		\begin{enumerate}[(a)]
			\item When $\sqrt{12r-3}\in\m{Q}$,  it follows from Proposition \ref{Prop, IE}(a) that $\sigma_r=1$. Therefore, $s_r=1$.
			\item This part is obviously due to the definition (\ref{sharp index by IE}) for $s_r$.
			\item When $r$ is an algebraic number, $\sqrt{12r-3}$ is also an algebraic number. Now if $\sqrt{12r-3}\notin\m{Q}$, then it follows from Proposition \ref{Prop, IE}(c) that $\sigma_r=2$. Hence, $s_r=\frac12$.
			\item By Proposition \ref{Prop, IE}(d), $\sigma_r=2$ for almost every $r\geq \frac14$. Thus, $s_{r}=\frac12$ for almost every $r\geq \frac14$.
			\item Combining Proposition \ref{Prop, IE}(e) with Proposition \ref{Prop, inv of IE}(a), we conclude the range of $\sigma_r$ is $\{1\}\cup [2,\infty]$. As a result, the range of $s_r$ is $[\frac12,1]$.
		\end{enumerate}
	\end{proof}
	
	Now we are ready to prove Theorem \ref{Thm, same index in Oh-paper}.

	\begin{proof}[Proof of Theorem \ref{Thm, same index in Oh-paper}]
		Let $a_2\in(0,4]\setminus\{1\}$. Recall $\nu_{c}=\max\{\nu(c_1),\nu(c_2)\}$, where 
		\[c_1=\frac12-\frac{1}{6}\sqrt{\frac{12}{a_2}-3}, \quad c_2=\frac12+\frac{1}{6}\sqrt{\frac{12}{a_2}-3},  \]
		and recall $\nu_{d}=\max\{\nu(d_1),\nu(d_2)\}$, where 
		\[d_{1}=\frac{-3a_2-\sqrt{3a_2(4-a_2)}}{2(1-a_2)}, \quad d_2=\frac{-3a_2+\sqrt{3a_2(4-a_2)}}{2(1-a_2)}.\]
		Denote $r=\frac{1}{a_2}$ and $\rho_{r}=\sqrt{12r-3}$. Then $r\geq \frac14$ and it follows from (\ref{sigma_r and s_r}) that $\sigma_r=\mu(\rho_r)$. In addition, based on Proposition \ref{Prop, IE},
		\[\mu(c_1)=\mu\bigg(\sqrt{\frac{12}{a_2}-3}\bigg)=\sigma_r.\]
		Similarly, $\mu(c_2)=\sigma_r$. On the other hand, noticing that $d_1=\frac{1}{c_1}$ and $d_2=\frac{1}{c_2}$, so it follows from Proposition \ref{Prop, inv of IE} that $\mu(d_1)=\mu(c_1)$ and $\mu(d_2)=\mu(c_2)$. Hence, 
		\be\label{same indexes}
		\mu(d_1)=\mu(d_2)=\mu(c_1)=\mu(c_2)=\sigma_r.\ee
		Then according to Proposition \ref{Prop, IE and MTI relation},
		$\nu(d_1)=\nu(d_2)=\nu(c_1)=\nu(c_2)$.
		In particular, $\nu_{d}=\nu_{c}=\nu(c_1)$.
		
		Next, we will justify (\ref{simple index}). Denote 
		\be\label{s tilde}
		\wt{s}=\min\Big\{1,\frac12+\frac12\nu(c_1)\Big\}.\ee
		Then it suffices to prove $\wt{s}=s_{r}$.
		
		\begin{itemize}
			\item Case 1: $\rho_{r}\in\m{Q}$. In this case, $c_1\in\m{Q}$, so it follows from Proposition \ref{Prop, IE and MTI relation} that $\nu(c_1)=\infty$. Therefore, $\wt{s}=1$. On the other hand, by Proposition \ref{Prop, s_r}(a), we also have $s_r=1$. 
			
			\item Case 2: $\rho_{r}\in\m{R}\setminus\m{Q}$. In this case, $c_1\in \m{R}\setminus\m{Q}$, so it follows from Proposition \ref{Prop, IE and MTI relation} that $\nu(c_1)=\mu(c_1)-2$, which implies $\nu(c_1)=\sigma_r-2$. Putting this into (\ref{s tilde}) yields
			\[\wt{s}=\min\Big\{1,\frac{\sigma_r-1}{2}\Big\}.\]
			Since $\rho_{r}\in\m{R}\setminus\m{Q}$, then Proposition \ref{Prop, IE}(b) implies $\sigma_r=\mu(\rho_r)\geq 2$.  Thus, we conclude from (\ref{sigma_r and s_r}) that 
			\[s_r=\min\Big\{1,\frac{\sigma_r-1}{2}\Big\}=\wt{s}.\] 
		\end{itemize}
	\end{proof}
	
	\section{Linear estimates}
	\label{Sec, lin est}
	
	Let $\psi\in C_{0}^{\infty}(\m{R})$ be a bump function supported on $[-2,2]$ and $\psi=1$ on $[-1,1]$.  We first present two linear  estimates, one is for the solution to the homogeneous linear KdV equation (\ref{linear eq}), and another one is for the solution to the forced linear KdV equation (\ref{linear eq}) with the right hand side being $F$ instead of 0. Recall the notation $S^{\a}_{\lam}$ as defined in (\ref{semigroup op}).
	
	\begin{lemma}\label{Lemma, lin est for KdV}
		There exists a constant $C$ which only depends on the bump function $\psi$ such that for any $\a,s,\lam\in\m{R}$ with $\a\neq 0$ and $\lam\geq 1$, 
		\be\label{lin est for KdV}
		\|\psi(t)S^{\a}_{\lam}(t)w_{0}\|_{Y^{\a}_{s,\lam}}\leq C\| w_{0}\|_{H^{s}(\m{T}_{\lam})}\ee
		and
		\be\label{lin est for Duhamel}
		\Big\Vert \psi(t)\int_{0}^{t}S^{\a}_{\lam}(t-t')F(t')dt' \Big\Vert_{Y^{\a}_{s,\lam}}\leq C\| F\|_{Z^{\a}_{s,\lam}}
		\ee
	\end{lemma}
	
	The proof of the above lemma is almost the same as those for Lemma 7.1 and Lemma 7.2 in \cite{CKSTT03}, so we omit it.  Next, we provide two well-known embedding results.
	
	\begin{lemma}\label{Lemma, L4t-2x}
		There exists a universal constant $C$ such that for any $\a\in\m{R}^*$, $\lam\geq 1$ and for any function $g$ on $\m{T}_{\lam}\times\m{R}$,
		\be\label{L4t-2x}
		\| g\|_{L^{4}_{t}L^{2}_{x}}\leq C\| g\|_{X^{\a}_{0,\frac{1}{4},\lam}}.\ee 
	\end{lemma}
	The proof of this estimate can be found in (\cite{Oh07}, Lemma 2.3.2).

	%

	\begin{lemma}\label{Lemma, L4 est}
		Let $\a\in\m{R}^{*}$. Then there exists a constant $C=C(\a)$ such that for any $\lam\geq 1$ and for any function $g$ on $\m{T}_{\lam}\times\m{R}$, 
		\be\label{L4 est}
		\| g\|_{L^{4}(\m{T}_{\lam}\times\m{R})}\leq C\| g\|_{X^{\a}_{0,\frac{1}{3},\lam}}.\ee
	\end{lemma}
	
	When $\alpha =1$ and $\lambda=1$, Lemma \ref{Lemma, L4 est}  was first proved by Bourgain in \cite{Bou93b} for a version when the left hand side of (\ref{L4 est}) is localized in time. Then Tao removed such a restriction in (\cite{Tao01}, Proposition 6.4). Later, a more elementary proof was provided by Oh in his online note \cite{Oh12}. Actually, similar method had been applied earlier to the Schr\"odinger equation (see e.g. \cite{Tao06}, Proposition 2.13).

	\section{Bilinear estimates in non-divergence form}
	\label{Sec, bilin est, non-div}
	This section will present the rigorous version of the bilinear estimates (\ref{nondiv 1}) and (\ref{nondiv 2}) in non-divergence form. Meanwhile, Theorem \ref{Thm, bilin est, non-div} will be broken into Lemma \ref{Lemma, bilin est, nondiv 1}, \ref{Lemma, bilin est, nondiv 2} and Proposition \ref{Prop, sharp bilin est, nondiv 1}, \ref{Prop, sharp bilin est, nondiv 2} in more general settings. We denote $s_r$ as  in (\ref{sharp index by IE}) for any $r\geq \frac14$.

	\begin{lemma}\label{Lemma, bilin est, nondiv 1}
		Let $\lam\geq 1$ and $\a_1,\a_2\in\m{R}^{*}$ with $r:=\frac{\a_2}{\a_1}$. Assume one of the conditions below is satisfied. 
		\begin{enumerate}[(a)]
			\item $r<\frac14$, $s\geq -\frac14$ and $p>0$;
			\item $r=1$, $s\geq \frac12$ and $p>0$;
			\item $r\in[\frac14,\infty)\setminus\{1\}$, $s\geq 1$ and $p>0$;
			\item $r\in[\frac14,\infty)\setminus\{1\}$ with $s_r<1$, $s>s_r$ and $p>s_r$.
		\end{enumerate}
		Then there exist constants $\eps=\eps(\a_1,\a_2)$ and $C=C(\a_1,\a_2,s,p)$  such that
		\be\label{bilin est, nondiv 1}
		\|(\p_{x}w_{1})w_{2}\|_{Z^{\a_{1}}_{s,\lam}}\leq C\,\lam^{p}\| w_{1}\|_{X_{s,\frac{1}{2},\lam}^{\a_{1}}}\| w_{2}\|_{X_{s,\frac{1}{2},\lam}^{\a_{2}}} \ee
		for any $w_1 \in X_{s,\frac{1}{2},\lam}^{\a_{1}}$ and $w_2 \in X_{s,\frac{1}{2},\lam}^{\a_{2}}$ with $\wh{w_2}(0,\cdot)=0$.
	\end{lemma}

	\begin{lemma}\label{Lemma, bilin est, nondiv 2}
		Let $\lam\geq 1$ and $\a_1,\a_2\in\m{R}^{*}$ with $r:=\frac{\a_2}{\a_1}$. Assume one of the conditions below is satisfied.
		\begin{enumerate}[(a)]
			\item $r<\frac14$ and $s\geq -\frac14$ and $p>0$;
			\item $r=1$, $s\geq \frac12$, $p>0$ and $\wh{w_1}(0,\cdot)=0$;
			\item $r\in[\frac14,\infty)\setminus\{1\}$, $s\geq 1$ and $p>0$;
			\item $r\in[\frac14,\infty)\setminus\{1\}$ with $s_r<1$, $s>s_r$ and $p>s_r$.
		\end{enumerate}
		Then there exist constants $\eps=\eps(\a_1,\a_2)$ and $C=C(\a_1,\a_2,s,p)$ such that 
		\be\label{bilin est, nondiv 2}
		\| w_{1}(\p_{x}w_{2})\|_{Z^{\a_{1}}_{s,\lam}}\leq C\,\lam^{p}\| w_{1}\|_{X_{s,\frac{1}{2},\lam}^{\a_{1}}}\| w_{2}\|_{X_{s,\frac{1}{2},\lam}^{\a_{2}}}\ee
		for any $w_1 \in X_{s,\frac{1}{2},\lam}^{\a_{1}}$ and $w_2 \in X_{s,\frac{1}{2},\lam}^{\a_{2}}$.
	\end{lemma}
	
	Next, we will address the sharpness of Lemma \ref{Lemma, bilin est, nondiv 1} and Lemma \ref{Lemma, bilin est, nondiv 2}. Without loss of generality, we take $\lam=1$.
	
	\begin{proposition}\label{Prop, sharp bilin est, nondiv 1}
		The bilinear estimate 
		\be\label{bilin est, nondiv 1, reg}
		\|(\p_{x}w_{1})w_{2}\|_{X^{\a_{1}}_{s,b-1}}\leq C\|w_{1}\|_{X_{s,b}^{\a_{1}}}\|w_{2}\|_{X_{s,b}^{\a_{2}}}, \quad \wh{w_2}(0,\cdot)=0,\ee
		fails for any $b\in\m{R}$ (and hence (\ref{bilin est, nondiv 1}) fails by taking $b=\frac12$) under any of the following conditions.
		\begin{itemize}
			\item[(a)] $r<\frac{1}{4}$ and $s<-\frac{1}{4}$;
			\item[(b)] $r=1$ and $s<\frac12$;
			\item[(c)] $r\in[\frac14,\infty)\setminus\{1\}$ and $s<s_r$;
			\item[(d)] $r\in\m{R}^{*}$, $s\in\m{R}$, but without the restriction $\wh{w_2}(0,\cdot)=0$ in (\ref{bilin est, nondiv 1, reg}).
		\end{itemize}
	\end{proposition}

	\begin{proposition}\label{Prop, sharp bilin est, nondiv 2}
		The bilinear estimate 
		\be\label{bilin est, nondiv 2, reg}
		\| w_{1}(\p_{x}w_{2})\|_{X^{\a_{1}}_{s,b-1}}\leq C\| w_{1}\|_{X_{s,b}^{\a_{1}}}\| w_{2}\|_{X_{s,b}^{\a_{2}}}\ee
		fails for any $b\in\m{R}$ (and hence (\ref{bilin est, nondiv 2}) fails by taking $b=\frac12$) under any of the following conditions.
		\begin{itemize}
			\item[(a)] $r<\frac{1}{4}$ and $s<-\frac{1}{4}$;
			\item[(b)] $r=1$, $s<\frac12$ and $\wh{w_1}(0,\cdot)=0$;
			\item[(c)] $r=1$, $s\in\m{R}$, but without the restriction $\wh{w_1}(0,\cdot)=0$;
			\item[(d)] $r\in[\frac14,\infty)\setminus\{1\}$ and $s<s_r$.
		\end{itemize}
	\end{proposition}

	\section{Proofs of the bilinear estimates}
	\label{Sec, proof of bilin}
	The goal of this section is to prove Lemma \ref{Lemma, bilin est, nondiv 1} and Lemma \ref{Lemma, bilin est, nondiv 2}. Since their proofs are similar, we will only justify Lemma \ref{Lemma, bilin est, nondiv 2}.  
	
	\subsection{Idea of the proof} 
	\label{Subsec, idea of bilin est}
	Without loss of generality,  we consider the following simpler version of (\ref{bilin est, nondiv 2}) with $\lam=1$,
	\be\label{bilin est, nd2, idea}
	\|w_{1}(\p_{x}w_{2})\|_{X^{\a_{1}}_{s,-\frac12}}\ls \| w_{1}\|_{X^{\a_{1}}_{s,\frac12}}\| w_{2}\|_{X^{\a_{2}}_{s,\frac12}}.\ee
	By duality and Plancherel identity, in order to verify (\ref{bilin est,  nd2, idea}), it suffices to prove (see \cite{Tao01} or Lemma \ref{Lemma, bilin to weighted l2})
	\be\label{weighted l2 form, nd2, idea}
	\i_{A}\frac{|k_{2}|\la k_{3}\ra^{s}\prod\limits_{i=1}^{3}|f_{i}(k_{i},\tau_{i})|}{\la k_{1}\ra^{s}\la k_{2}\ra^{s}\la L_{1}\ra^{\frac12}\la L_{2}\ra^{\frac12}\la L_{3}\ra^{\frac12}} \leq C\,\prod_{i=1}^{3}\|f_{i}\|_{L^{2}(\m{Z}\times\m{R})}, \quad\forall\,\{f_{i}\}_{1\leq i\leq 3},\ee
	where $\la\cdot\ra=1+|\cdot|$,  $A=\Big\{(k_1,k_2,k_3,\tau_1,\tau_2,\tau_3)\in \m{Z}^{3}\times\m{R}^{3}:\sum\limits_{i=1}^{3}k_{i}=\sum\limits_{i=1}^{3}\tau_{i}=0\Big\}$ and 
	\[L_{1}=\tau_{1}-\phi^{\a_{1}}(k_{1}),\quad L_{2}=\tau_{2}-\phi^{\a_{2}}(k_{2}), \quad L_{3}=\tau_{3}-\phi^{\a_{1}}(k_{3}),\]
	where $\phi^{\a}(k)=\a k^3$ is  as defined  in (\ref{phase fn}). 
	
	In (\ref{weighted l2 form, nd2, idea}), the loss of the spatial derivative in the bilinear estimate (\ref{bilin est, nd2, idea}) is reflected via the term $\dfrac{|k_{2}|\la k_{3}\ra^{s}}{\la k_{1}\ra^{s}\la k_{2}\ra^{s}}$ and the gain of the time derivative is reflected via the term $\la L_{1}\ra^{\frac12}\la L_{2}\ra^{\frac12}\la L_{3}\ra^{\frac12}$. How to compensate the loss of the spatial derivative from the gain of the time derivative is the key point. Denote 
	\[R_{1}=\frac{|k_{2}|\la k_{3}\ra^{s}}{\la k_{1}\ra^{s}\la k_{2}\ra^{s}}\quad\text{and}\quad R_{2}=\la L_{1}\ra^{\frac12}\la L_{2}\ra^{\frac12}\la L_{3}\ra^{\frac12}.\]
	Then we need to control $R_{1}$ by $R_{2}$. Since $\sum\limits_{i=1}^{3}k_{i}=0$, then $\la k_{3}\ra\leq \la k_{1}\ra\la k_{2}\ra$. As a result, $R_{1}$ is decreasing in $s$, which means the larger $s$ is, the more likely the bilinear estimate will hold. So the interest lies in the search for the smallest $s$ such that the bilinear estimate holds. Noticing that $L_{i}$ contains the time variable $\tau_{i}$, a single $L_{i}$ alone  can barely have any contributions to control $R_1$.  On the other hand, as  $\sum\limits_{i=1}^{3}\tau_{i}=0$, 
	\[\sum_{i=1}^{3}L_{i}=-\big[\phi^{\a_1}(k_1)+\phi^{\a_2}(k_2)+\phi^{\a_1}(k_3)\big],\]
	is a function of $\{k_{i}\}_{i=1,2,3}$ only.  Because of this,  we define the resonance function $H_2$ as in Definition \ref{Def, res fcn}:  
	\be\label{res fn, nd2, idea}
	H_2(k_{1},k_{2},k_{3}):=\phi^{\a_1}(k_1)+\phi^{\a_2}(k_2)+\phi^{\a_1}(k_3).\ee
	Since $R_{2}\gs \la H_2\ra^{\frac12}$, then it is easier to control $R_1$ by $R_2$ in the region where $H_2$ is large.  The situation becomes more complicated near the region where $H_2$ vanishes.  The zero set of $H_2$ is called  the {\it resonance set} as in Definition \ref{Def, res fcn}. 
	
	By writing $k_{3}=-k_1-k_2$ in (\ref{res fn, nd2, idea}) and simplifying,
	\be\label{H2}
	H_{2}(k_1,k_2,k_3)=(\a_2-\a_1)k_{2}^{3}-3\a_{1}k_{1}k_{2}^{2}-3\a_{1}k_{1}^{2}k_{2}, \quad\forall\, \sum_{i=1}^{3}k_{i}=0.\ee
	If $k_{2}=0$, then $H_{2}(k_1,k_2,k_3)=0$. If $k_{2}\neq 0$, then $H_{2}$ can be rewritten as
	\be\label{H2 compact}
	H_{2}(k_1,k_2,k_3)=-3\a_{1}k_{2}^{3}\,h_{r}\Big(\frac{k_1}{k_2}\Big), \quad\forall\, \sum_{i=1}^{3}k_{i}=0 \text{\;with\;} k_2\neq 0,\ee
	where 
	\be\label{h}
	h_{r}(x)=x^{2}+x+\frac{1-r}{3}.\ee
	The following is the classification of the roots of $h_{r}$ depending on the values of $r$.
	\begin{enumerate}[(1)]
		\item $r<\frac14$: $h_{r}(x)$ does not have real roots;
		
		\item $r=\frac14$: $h_{r}(x)$ has a unique root $-\frac12$;
		
		\item $r>\frac14$ but $ r\neq 1$: $h_{r}(x)$ has two roots, neither of which equals -1 or 0;
		
		\item $r=1$: $h_{r}(x)$ has two roots -1 and 0.
	\end{enumerate}
	
	Among the above four situations, Case (3) is the most interesting one, so we will first focus on this case.  Assume $r\in\big(\frac14,\infty\big)\setminus\{1\}$ and denote the two roots of $h_{r}(x)$ as $x_{1r}$ and $x_{2r}$, i.e.
	\be\label{roots}
	x_{1r}:=-\frac12-\frac{\sqrt{12r-3}}{6},\quad   x_{2r}:=-\frac12+\frac{\sqrt{12r-3}}{6}.\ee
	Define $\sigma_{r}$ and $s_{r}$ as in (\ref{sharp index by IE}). Then it follows from Proposition \ref{Prop, inv of IE} that 
	\be\label{IE of root}
	\mu(x_{1r})=\mu(x_{2r})=\mu(\sqrt{12r-3})=\sigma_{r}.\ee
	In addition, the resonance function $H_2$ can be written as 
	\[H_{2}(k_{1},k_{2},k_{3})=-3\a_1 k_2^3\Big(\frac{k_1}{k_2}-x_{1r}\Big)\Big(\frac{k_1}{k_2}-x_{2r}\Big).\]
	As a result, the resonance set consists of three lines: $k_{2}=0$, $k_{1}=x_{1r}k_{2}$ and $k_{1}=x_{2r}k_{2}$. The most difficult estimate is near the line $k_{1}=x_{1r}k_{2}$ or $k_{1}=x_{2r}k_{2}$. Without loss of generality, let us consider the region near the line $ k_1=x_{1r}k_2 $. In this situation, $\big|\frac{k_1}{k_2}-x_{2r}\big|\approx |x_{1r}-x_{2r}|$ which is a positive constant. Hence, 
	\be\label{H2, res, intro}
	|H_2(k_1,k_2,k_3)|\sim |k_2|^{3}\Big|x_{1r}-\frac{k_1}{k_2}\Big|.\ee
	In addition, when $\frac{k_1}{k_2}$ is very close to $x_{1r}$, we have $|k_1|\sim |k_2|\sim |k_3|$ . Consequently, 
	\be\label{loss of spatial deriv}
	R_{1}\sim |k_2|^{1-s}.\ee
	On the other hand, noticing that both $k_1$ and $k_2$ are integers, the estimate of $\big|x_{1r}-\frac{k_1}{k_2}\big|$ reduces to the problem of the Diophantine approximation of $x_{1r}$. 
	\begin{itemize}
		\item If $x_{1r}\in\m{Q}$ or if its irrationality exponent $\sigma_{r}=\mu(x_{1r})\geq 3$, then $s_r=1$ and there exist infinitely many $(k_1,k_2)$ such that 
		\[\Big|x_{1r}-\frac{k_1}{k_2}\Big|=0 \text{\;\; or \;\;} \Big|x_{1r}-\frac{k_1}{k_2}\Big|\ls \frac{1}{k_2^3}.\]
		So $\la H_2(k_1,k_2,k_3)\ra\sim 1$ due to (\ref{H2, res, intro}). Then in order to have $|R_1|\ls \la H_2\ra^{\frac12}$, it follows from (\ref{loss of spatial deriv}) that $s\geq 1=s_r$.

		\item If $2\leq \sigma_r<3$, then $s_r=\frac{\sigma_r-1}{2}$ and it follows from Lemma \ref{Lemma, bdd approx beyond IE} that for any $\eps>0$ and for any integers $k_1$ and $k_2$,
		\[\Big|x_{1r}-\frac{k_1}{k_2}\Big|\gs \frac{1}{|k_2|^{\sigma_r+\eps}}.\]
		So it follows from (\ref{H2, res, intro}) that $|H_2(k_1,k_2,k_3)|\gs |k_2|^{3-\sigma_r-\eps}$. Then in order to have $|R_1|\ls \la H_2\ra^{\frac12}$, we need $1-s\leq (3-\sigma_r-\eps)/2$, that  is 
		\[s\geq \frac{\sigma_r-1}{2}+\frac{\eps}{2}=s_r+\frac{\eps}{2}.\]
	\end{itemize}
	The above argument explains why the critical index  is $s_{r}$ when $r\in(\frac14,\infty)\setminus\{1\}$.  Next, we will briefly discuss the rest cases (1), (2) and (4). 
	\begin{itemize}
		\item When $r=\frac14$,  the argument is similar to the above. The two roots of $h_r(x)$ are the same: $x_{1r}=x_{2r}=-\frac12$. So $s_r=\sigma_r=1$. Meanwhile, $H_2(k_1,k_2,k_3)=-3\a_1 k_2^3\big(\frac{k_1}{k_2}+\frac12\big)^2$. So there exist infinitely many integer pairs $(k_1,k_2)$ such that $\frac{k_1}{k_2}+\frac12=0$, which implies $|H_2(k_1,k_2,k_3)|=0$. In order to have $|R_1|\ls \la H_2\ra^{\frac12}$, it follows from (\ref{loss of spatial deriv}) that $s\geq 1=s_r$.
		
		\item When $r<\frac14$, $|h_r(x)|\gs 1+x^2$, so it follows from (\ref{H2 compact}) that 
		\be\label{lower bdd for H2, intro}
		|H_2(k_1,k_2,k_3)|\gs |k_2|\sum_{i=1}^{3}k_i^2.\ee
		For $s<0$, the worst situation is when $|k_1|\sim |k_2|\gg 1$ and $|k_3|\ls 1$. In this situation, $R_1\sim |k_2|^{1-2s}$ and $|H_2|\sim |k_2|^{3}$. So in order to ensure $R_1\ls \la H_2\ra^{\frac12}$, $s$ needs to be at least $-\frac14$.
		
		\item When $r=1$, $|H_2(k_1,k_2,k_3)|\sim |k_1 k_2 k_3|$.  The worst situation occurs when $k_1=-k_2$ and $k_3=0$, which implies $H_2=0$ and $R_1\sim |k_2|^{1-2s}$. So in order to have $R_1\ls \la H_2\ra^{\frac12}=1$, $s$ should be at least $\frac12$.
	\end{itemize}

	\subsection{Auxiliary results}
	For any vector $(\vec{k},\vec{\tau})\in \m{Z}_{\lam}^{3}\times\m{R}^{3}$, we denote it as $(\vec{k},\vec{\tau})=(k_{1},k_{2},k_{3},\tau_{1},\tau_{2},\tau_{3})$. For any $\lam \geq 1$,  
	\be\label{int domain}
	A_{\lam}: =\Big\{(\vec{k},\vec{\tau})\in \m{Z}_{\lam}^{3}\times\m{R}^{3}:\sum_{i=1}^{3}k_{i}=\sum_{i=1}^{3}\tau_{i}=0\Big\},\ee
	and for any given $(k_{3},\tau_{3})\in \m{Z}_{\lam}\times\m{R}$, 
	\be\label{int domain, line}
	A_{\lam}(k_{3},\tau_{3}):=\big\{(k_{1},k_{2},\tau_{1},\tau_{2})\in \m{Z}_{\lam}^{2}\times\m{R}^{2}: (\vec{k},\vec{\tau})\in A_{\lam}\big\}.\ee

	\begin{lemma}\label{Lemma, bilin to weighted l2}
		Let $\lam\geq 1$, $s\in\m{R}$ and $\a_{i}\in \m{R}^{*}$ for $1\leq i\leq 3$. The bilinear estimate 
		\be\label{bilin, general}
		\|w_{1}(\p_{x}w_{2})\|_{Z_{s,\lam}^{\a_{3}}}\leq C\lam^{p}\|w_{1}\|_{X^{\a_{1}}_{s,\frac{1}{2},\lam}}\|w_{2}\|_{X^{\a_{2}}_{s,\frac{1}{2},\lam}}, \quad\forall\, \{w_{i}\}_{i=1,2}\ee
		holds if and only if the following two estimates hold, 
		\be\label{weighted l2, general, reg}
		\i_{A_{\lam}}\frac{k_{2}\la k_{3}\ra^{s}}{\la k_{1}\ra^{s}\la k_{2}\ra^{s}}\prod_{i=1}^{3}\frac{f_{i}(k_{i},\tau_{i})}{\la L_{i}\ra^{\frac{1}{2}}}\leq C\lam^{p}\prod_{i=1}^{3}\|f_{i}\|_{L^{2}(\m{Z}_{\lam}\times\m{R})}, \quad\forall\, \{f_{i}\}_{i=1,2,3}\ee
		and 
		\be\label{weighted l2, general, aux}
		\bigg\| \frac{1}{\la L_{3}\ra}\i_{A_{\lam}(k_{3},\tau_{3})}\frac{k_{2}\la k_{3}\ra^{s}}{\la k_{1}\ra^{s}\la k_{2}\ra^{s}}\prod_{i=1}^{2}\frac{f_{i}(k_{i},\tau_{i})}{\la L_{i}\ra^{\frac{1}{2}}}\bigg\|_{L^{2}_{k_{3}}L^{1}_{\tau_{3}}}
		\leq C\lam^{p}\prod_{i=1}^{2}\| f_{i}\|_{L^{2}(\m{Z}_{\lam}\times\m{R})}, \quad\forall\, \{f_{i}\}_{i=1,2},\ee
		where $L_{i}=\tau_{i}-\phi^{\a_{i}}(k_{i})$.
	\end{lemma}
	
	This lemma is essentially proved in \cite{Tao01, CKSTT03} by using duality, Plancherel identity and definition (\ref{Z space}), so we omit the details here. 
	
	\begin{lemma}\label{Lemma, key}
		Let $\lam\geq 1$, $q\geq \frac13$ and $\a_i\in\m{R}^*$ for $1\leq i\leq 3$. Then there exists a constant $C=C(\a_{1},\a_{2},\a_{3},p)$ such that for any functions $\{f_{i}\}_{i=1}^{3}$ on $\m{Z}_{\lam}\times\m{R}$, 
		\be\label{key lemma-reg}
		\i_{A_{\lam}} M^{q}\prod_{i=1}^{3}\frac{|f_{i}(k_{i},\tau_{i})|}{\la L_{i}\ra^{q}}\leq C\prod_{i=1}^{3}\|f_{i}\|_{L^{2}(\m{Z}_{\lam}\times\m{R})}, \ee
		where $L_{i}=\tau_{i}-\phi^{\a_{i}}(k_{i})$ and $M:=\max\big\{\la L_1\ra,\la L_2\ra, \la L_3\ra\big\}$. In particular, if $H(k_1,k_2,k_3)$ denotes the resonance function as defined in Definition \ref{Def, res fcn}, then 
		\[\i_{A_{\lam}} \la H(k_1,k_2,k_3)\ra^{q}\prod_{i=1}^{3}\frac{|f_{i}(k_{i},\tau_{i})|}{\la L_{i}\ra^{q}}\leq C\prod_{i=1}^{3}\|f_{i}\|_{L^{2}(\m{Z}_{\lam}\times\m{R})}.\]
	\end{lemma}
	
	This lemma is also standard (e.g. see \cite{Bou93b, CKSTT03, Oh09}),  but the statement is a little bit different from that in the literature, so we will briefly include a proof here for the convenience of the readers.
	
	\begin{proof}[Proof of Lemma \ref{Lemma, key}]
		Since $H(k_1,k_2,k_3)=-\sum\limits_{i=1}^{3}L_i$, it is obvious that $\la H\ra\ls M$, so it suffices to prove (\ref{key lemma-reg}).
		\begin{itemize}
			\item Let's first focus on the region where $ M=\la L_1\ra $. In this region, $ \la L_1\ra^{q}\gs  M^{q}$. Define $ g_1=\mcal{F}^{-1}\big(f_1(k_1,\tau_1)\big) $ and
			\[ g_i=\mcal{F}^{-1}\bigg(\frac{|f_i(k_i,\tau_i)|}{\la L_i\ra^{q}}\bigg), \quad i=2,3. \]
			Then 
			\begin{align*}
				\text{LHS of (\ref{key lemma-reg})} &\ls \int [\mcal{F}(g_1)](k_1,\tau_1)[\mcal{F}(g_2)](k_2,\tau_2)[\mcal{F}(g_3)](k_3,\tau_3)\\
				&\ls \| g_1\|_{L^2(\m{T}_{\lam}\times\m{R})} \| g_2\|_{L^4(\m{T}_{\lam}\times\m{R})} \| g_3\|_{L^4(\m{T}_{\lam}\times\m{R})}.
			\end{align*}
			As a result, it follows from Lemma \ref{Lemma, L4 est} and $q\geq \frac13$ that 
			\[ \| g_1\|_{L^2} \| g_2\|_{L^4} \| g_3\|_{L^4}\ls \| g_1\|_{L^2}\| g_2\|_{X^{\a_2}_{0,\frac13,\lam}} \| g_3\|_{X^{\a_3}_{0,\frac13,\lam}}\ls \prod_{i=1}^{3}\| f\|_{L^2(\m{Z}_{\lam}\times\m{R})}.\]
			
			\item In the regions where $ M=\la L_2\ra $ or $ \la L_3\ra $, the argument is similar, so (\ref{key lemma-reg}) is justified.
		\end{itemize}
	\end{proof}

	\subsection{Proof of Lemma \ref{Lemma, bilin est, nondiv 2}}
	\label{Subsec, bilin, nondiv 2}
	As we have seen in Section \ref{Subsec, idea of bilin est}, the main ideas in the proofs for part (a)--part (d) in Lemma \ref{Lemma, bilin est, nondiv 2} are analogous. So we will only carry out the detailed proof for part (d) which is the most technical case.  The framework of the following proof is similar to that in \cite{Oh09}.

	\begin{proof}[\bf Proof of Part (d)]
		First, we recall from Proposition \ref{Prop, s_r}(b) that $s_{r}\geq \frac12$. In addition, it follows from  
		the assumption $s_r< 1$ and Proposition \ref{Prop, s_r}(a) that $r\neq \frac14$. The case of $\frac14<r<1$ is analogous to the case of $r>1$, so we will just assume $\frac14<r<1$ in the rest of the proof.
		
		Fix $s>s_r$ and $p>s_r$, according to Lemma \ref{Lemma, bilin to weighted l2}, it remains to prove
		\be\label{nd 2-d, l2-reg}
		\i_{A_{\lam}}\frac{|k_{2}|\la k_{3}\ra^{s}}{\la k_{1}\ra^{s}\la k_{2}\ra^{s}}\prod_{i=1}^{3}\frac{|f_{i}(k_{i},\tau_{i})|}{\la L_{i}\ra^{\frac{1}{2}}}\leq C\lam^{p}\prod_{i=1}^{3}\|f_{i}\|_{L^{2}(\m{Z}_{\lam}\times\m{R})}, \ee
		and 
		\be\label{nd 2-d, l2-aux}
		\bigg\| \frac{1}{\la L_{3}\ra}\i_{A_{\lam}(k_{3},\tau_{3})}\frac{|k_{2}|\la k_{3}\ra^{s}}{\la k_{1}\ra^{s}\la k_{2}\ra^{s}}\prod_{i=1}^{2}\frac{|f_{i}(k_{i},\tau_{i})|}{\la L_{i}\ra^{\frac{1}{2}}}\bigg\|_{L^{2}_{k_{3}}L^{1}_{\tau_{3}}}
		\leq C\lam^{p}\prod_{i=1}^{2}\|f_{i}\|_{L^{2}(\m{Z}_{\lam}\times\m{R})},\ee
		where 
		\[L_{1}=\tau_{1}-\phi_1(k_{1}),\quad L_{2}=\tau_{2}-\phi_2(k_{2}),\quad  L_{3}=\tau_{3}-\phi_1(k_{3}).\]
		Since $\frac{\la k_3\ra}{\la k_1\ra\la k_2\ra}\leq 1$, it suffices to consider the case when $s$ is sufficiently close to $s_r$. So we can just assume 
		\be\label{order or power}
		s_r<s<p<1. \ee
		Next, we will prove (\ref{nd 2-d, l2-reg}) first and then briefly mention the proof for (\ref{nd 2-d, l2-aux}).
		
		{\bf Proof of (\ref{nd 2-d, l2-reg})}.  First, in the region where $k_2=0$, the integrand on the LHS of (\ref{nd 2-d, l2-reg}) vanishes, so we only need to focus on the region where $k_2\neq 0$.  Recall the resonance function $H_2$ in (\ref{H2 compact}):
		\be\label{nd 2-d, res}
		H_{2}(k_1,k_2,k_3)=-3\a_{1}k_{2}^{3}\,h_{r}\Big(\frac{k_1}{k_2}\Big), \quad\forall\, \sum_{i=1}^{3}k_{i}=0 \text{\; with \;} k_2\neq 0,\ee
		where $h_{r}$ is defined in (\ref{h}). When $\frac14<r<1$, it follows from (\ref{roots}) that the two roots $x_{1r}$ and  $x_{2r}$ of $h_{r}$ satisfy $-1<x_{1r}<x_{2r}<0$. Hence, there exists a constant 
		$d_{r}\in\big(0,\frac18\big)$, which only depends on $r$, such that (see Figure \ref{Fig, roots})
		\be\label{location of roots} -1<x_{1r}-2d_r<x_{1r}+2d_r<x_{2r}-2d_r<x_{2r}+2d_r<0.\ee
		\begin{figure}[htbp]
			\centering
			\begin{tikzpicture}[scale=0.8]
				\draw[->] (-18,0) -- (1,0) node[anchor=north] {};
				\foreach \x in {-17,-14,-12,-10,-7,-5,-3,0}
				\draw(\x,0.1)--(\x,-0.1) node[below] {};
				\draw (-17,-0.1) node[below] {$-1$};
				\draw (-14,-0.1) node[below] {$x_{1r}-2d_{r}$};
				\draw (-12,-0.2) node[below] {$x_{1r}$};
				\draw (-10,-0.1) node[below] {$x_{1r}+2d_{r}$};
				\draw (-7,-0.1) node[below] {$x_{2r}-2d_{r}$};
				\draw (-5,-0.2) node[below] {$x_{2r}$};
				\draw (-3,-0.1) node[below] {$x_{2r}+2d_{r}$};
				\draw (0,-0.1) node[below] {$0$};
			\end{tikzpicture}
			\caption{location of $x_{1r}$ and $x_{2r}$}
			\label{Fig, roots}
		\end{figure}
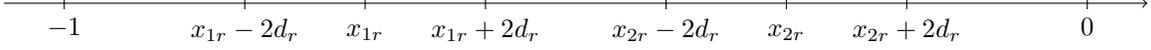
		
		\noindent On the other hand, there exists a constant $\eps_{r}$, which also only depends on $r$, such that 
		\be\label{large h}
		|h_{r}(x)|\geq \eps_{r}(1+x^2) \text{\; if \;} |x-x_{1r}|\geq d_{r} \text{\; and \;} |x-x_{2r}|\geq d_{r}.\ee
		In the following, we denote 
		\[\text{MAX}=\max\big\{\la L_{1}\ra, \la L_2\ra, \la L_3\ra\big\}.\]
		Since $|H_2(k_1,k_2,k_3)|=\big|\sum\limits_{i=1}^{3}L_{i}\big|$, then $\la H_2(k_1,k_2,k_3)\ra\ls \text{MAX}$.

		{\bf Case 1.}  $\Big|\dfrac{k_1}{k_2}-x_{1r}\Big|\geq d_r$ and $\Big|\dfrac{k_1}{k_2}-x_{2r}\Big|\geq d_r$. In this case, it follows from (\ref{nd 2-d, res}) and (\ref{large h}) that
		\[|H_{2}(k_1,k_2,k_3)|\sim |k_{2}|^{3}\,\Big| h_{r}\Big(\frac{k_1}{k_2}\Big)\Big|\gs |k_2|(k_1^2+k_2^2). \]
		So $|k_2|^3\ls |H_2|\ls \text{MAX} $. In addition, $k_2\neq 0$ implies that $|k_2|\geq \frac{1}{\lam}$. Hence, $|k_2|^2\ls \lam (\text{MAX})$.  Recalling $\dfrac{\la k_{3}\ra^{s}}{\la k_{1}\ra^{s}\la k_{2}\ra^{s}}\leq 1$, so
		\begin{align*}
			\text{LHS of (\ref{nd 2-d, l2-reg})} \ls \lam^{\frac12}\i_{}(\text{MAX})^{\frac12}\prod_{i=1}^{3}\frac{|f_{i}(k_{i},\tau_{i})|}{\la L_{i}\ra^{\frac{1}{2}}}\ls \lam^{\frac12}\prod_{i=1}^{3}\|f_{i}\|_{L^{2}(\m{Z}_{\lam}\times\m{R})},
		\end{align*}
		where the last inequality is due to Lemma \ref{Lemma, key}. This verifies (\ref{nd 2-d, l2-reg}) since $p>s_{r}\geq \frac12$.
		
		{\bf Case 2.} $\Big|\dfrac{k_1}{k_2}-x_{1r}\Big|< d_r$ or $\Big|\dfrac{k_1}{k_2}-x_{2r}\Big|< d_r$. Without loss of generality, it is assumed that $\Big|\dfrac{k_1}{k_2}-x_{1r}\Big|< d_r$. Then it can be seen from (\ref{location of roots}) or Figure \ref{Fig, roots} that $|k_1|\sim |k_2|\sim |k_3|$ and $\Big|\dfrac{k_1}{k_2}-x_{2r}\Big|> 3d_r$. Meanwhile, if $|k_2|\ls 1$, then (\ref{nd 2-d, l2-reg}) is obviously true. So we will just assume $|k_2|\gg 1$. Thus, (\ref{nd 2-d, l2-reg}) reduces to 
		\be\label{nd 2-d, res-reg, reduced}
		\i_{}|k_{2}|^{1-s}\prod_{i=1}^{3}\frac{|f_{i}(k_{i},\tau_{i})|}{\la L_{i}\ra^{\frac{1}{2}}}\leq C\lam^{p}\prod_{i=1}^{3}\|f_{i}\|_{L^{2}(\m{Z}_{\lam}\times\m{R})}.\ee
		Denote 
		\be\label{extra reg}\eps_{s}=2s-2s_r,\ee
		then $\eps_{s}>0$.  Next, we will split the range of $\big|\frac{k_1}{k_2}-x_{1r}\big|$ further.
		
		{\bf Case 2.1.} $\frac{d_r}{\lam |k_2|}\leq \big|\frac{k_1}{k_2}-x_{1r}\big|<d_{r}$.  In this case,  
		\[|H_{2}(k_1,k_2,k_3)|\sim |k_2|^{3}\Big|\frac{k_1}{k_2}-x_{1r}\Big|\Big|\frac{k_1}{k_2}-x_{2r}\Big|\gs \frac{k_2^2}{\lam}.\]
		Then it follows from $|H_{2}|\ls \text{MAX}$ and $s>s_{r}\geq \frac12$ that $|k_{2}|^{1-s}\ls (\lam |H_{2}|)^{\frac{1-s}{2}}\ls \lam^{\frac14} (\text{MAX})^{\frac14}$. Hence, 
		\[\text{LHS of (\ref{nd 2-d, res-reg, reduced})}\ls \lam^{\frac14}\i_{}(\text{MAX})^{\frac14}\prod_{i=1}^{3}\frac{|f_{i}(k_{i},\tau_{i})|}{\la L_{i}\ra^{\frac{1}{2}}}\ls C\lam^{\frac14}\prod_{i=1}^{3}\|f_{i}\|_{L^{2}(\m{Z}_{\lam}\times\m{R})}\ls \text{RHS of (\ref{nd 2-d, res-reg, reduced})}.\]

		{\bf Case 2.2.} $\big|\frac{k_1}{k_2}-x_{1r}\big|< \frac{d_{r}}{\lam |k_2|}$. 
		
		Since $s_r<1$, it follows from (\ref{sharp index by IE}) that $s_r=\frac{\sigma_r-1}{2}$. In addition, by taking advantage of (\ref{IE of root}) and Lemma \ref{Lemma, bdd approx beyond IE}, there exists a constant $K=K(r,s)$ such that
		\[\Big| x_{1r}-\frac{k_1}{k_2}\Big|=\Big| x_{1r}-\frac{\lam k_1}{\lam k_2}\Big|\geq \frac{K}{|\lam k_2|^{\sigma_{r}+\eps_{s}}}.\]
		Combining $s_r=\frac{\sigma_r-1}{2}$  and (\ref{extra reg}) yields $\sigma_r+\eps_{s}=2s+1$. Therefore,
		\be\label{dist to root}
		\Big| x_{1r}-\frac{k_1}{k_2}\Big|\gs \frac{1}{|\lam k_2|^{2s+1}}\ee
		which implies
		\[|H_{2}(k_1,k_2,k_3)|\sim |k_2|^{3}\Big|\frac{k_1}{k_2}-x_{1r}\Big|\Big|\frac{k_1}{k_2}-x_{2r}\Big|\gs \lam^{-1-2s}|k_2|^{2-2s}.\]
		Then we obtain 
		\be\label{upper bdd for k_2}
		|k_2|^{1-s}\ls \lam^{\frac12+s}|H_2|^{\frac12}\ls \lam^{\frac12+s}(\text{MAX})^{\frac12}.\ee
		So (\ref{nd 2-d, res-reg, reduced}) reduces to 
		\be\label{nd 2-d, reduced, 2.2, reg}
		\lam^{\frac12}\i_{}(\text{MAX})^{\frac12}\prod_{i=1}^{3}\frac{|f_{i}(k_{i},\tau_{i})|}{\la L_{i}\ra^{\frac{1}{2}}}\leq C\lam^{p-s}\prod_{i=1}^{3}\|f_{i}\|_{L^{2}(\m{Z}_{\lam}\times\m{R})}.\ee
		
		\begin{itemize}
			\item {\bf Case 2.2.1.} $\la L_1\ra=\text{MAX}$. Then (\ref{nd 2-d, reduced, 2.2, reg}) reduces to 
			\be\label{nd 2-d, case 2.2, L_1 max}
			\lam^{\frac12}\i_{}\frac{1}{\la L_2\ra^{\frac12}\la L_3\ra^{\frac12}}\prod_{i=1}^{3}|f_{i}(k_{i},\tau_{i})|\leq C\lam^{p-s}\prod_{i=1}^{3}\|f_{i}\|_{L^{2}(\m{Z}_{\lam}\times\m{R})}.\ee
			For fixed $k_1$, it follows from $\big|\frac{k_1}{k_2}-x_{1r}\big|<\frac{d_r}{\lam |k_2|}$ that $k_{2}\in E_{r}(k_1)$ which is defined as follows.
			\be\label{set of k_2}
			E_{r}(k_1)=\Big\{k_2\in\m{Z}_{\lam}:\Big|k_2-\frac{k_1}{x_{1r}}\Big|<\frac{d_r}{|x_{1r}|\,\lam}\Big\}.\ee
			It is easily seen that the size $|E_{r}(k_1)|$ is at most $\frac{2d_r}{|x_{1r}|}$. Define 
			\be\label{res, g}
			g_{1}(k_1,\tau_{1})=|f_{1}(k_1,\tau_1)|,\quad g_{2}(k_2,\tau_2)=\frac{|f_{2}(k_2,\tau_2)|}{\la L_{2}\ra^{\frac12}},\quad g_{3}(k_3,\tau_3)=\frac{|f_{3}(k_3,\tau_3)|}{\la L_{3}\ra^{\frac12}}.\ee
			Since $(k_3,\tau_3)=-(k_1+k_2,\tau_1+\tau_2)$, the LHS of (\ref{nd 2-d, case 2.2, L_1 max}) is bounded above by
			\begin{eqnarray*}
				\lam^{\frac12}\int_{\m{Z}_{\lam}}\int_{\m{R}}g_{1}(k_1,\tau_1)\,\bigg(\int_{\m{Z}_{\lam}}\mb{1}_{E_{r}(k_1)}(k_2)\,\big[g_{2}(k_2,\cdot)\ast_{\tau} g_{3}(-k_1-k_2,\cdot)\big](-\tau_1)\,dk_2^{\lam}\bigg)\,d\tau_1\,dk_1^{\lam}.
			\end{eqnarray*}
			By the change of variable $(k_1,\tau_1)\rightarrow (-k_1,-\tau_1)$ and the definition (\ref{measure}) for the measure $dk_2^{\lam}$, the above quantity is seen to be dominated by 
			\be\label{bdd 3}
			\lam^{-\frac12}\|g_1\|_{L^{2}_{k_1}L^{2}_{\tau_1}}\Big\|\sum_{k_2\in E_{r}(-k_1)}\big[g_{2}(k_2,\cdot)\ast_{\tau} g_{3}(k_1-k_2,\cdot)\big](\tau_1)\Big\|_{L^{2}_{k_1}L^{2}_{\tau_1}}.\ee
			Since the size of the set $E_{r}(-k_1)$ is bounded by the constant $\frac{2d_r}{|x_{1r}|}$, then by applying the Plancherel identity, we obtain 
			
			\be\label{bdd 4}
			(\ref{bdd 3})\ls \lam^{-\frac12}\|g_1\|_{L^{2}_{k_1}L^{2}_{\tau_1}}\Big\|\sup_{k_2\in E_{r}(-k_1)} \big|\mcal{F}^{-1}_{t}g_{2}(k_2,t)\,\mcal{F}^{-1}_{t}g_{3}(k_1-k_2,t)\big|\Big\|_{L^{2}_{k_1}L^{2}_{t}}.\ee
			
			Noticing $\mcal{F}_{t}^{-1}=\mcal{F}_{x}\mcal{F}^{-1}$, then it follows from the definition of $\mcal{F}_x$ and Holder's inequality that 
			\be\label{sup of g_2}
			\big|\mcal{F}^{-1}_{t}g_{2}(k_2,t)\big|\ls \big\|\mcal{F}^{-1}g_{2}(x,t)\big\|_{L^{1}_{x}}\ls \lam^{\frac12}\big\|\mcal{F}^{-1}g_{2}(x,t)\big\|_{L^{2}_{x}}.\ee
			
			In addition, it is not difficult to show that for any fixed $t$, 
			\be\label{tedious est}
			\Big\|\sup_{k_2\in E_{r}(-k_1)}\big| \mcal{F}^{-1}_{t}g_{3}(k_1-k_2,t)\big|\Big\|_{L^{2}_{k_1}} \ls \big\|\mcal{F}^{-1}g_3(x,t)\big\|_{L^{2}_{x}}.\ee
			
			Then plugging (\ref{sup of g_2}) and (\ref{tedious est}) into (\ref{bdd 4}) yields 
			\[\text{RHS of (\ref{bdd 4})} \ls \|g_1\|_{L^{2}_{k_1}L^{2}_{\tau_1}}\Big\|\big\|\mcal{F}^{-1}g_{2}(x,t)\big\|_{L^{2}_{x}}\big\|\mcal{F}^{-1}g_3(x,t)\big\|_{L^{2}_{x}}\Big\|_{L^{2}_{t}}.
			\]
			Finally, by Holder's inequality and Lemma \ref{Lemma, L4t-2x},
			\begin{eqnarray*}
				\text{RHS of (\ref{bdd 4})} &\ls & \|g_1\|_{L^{2}_{k_1}L^{2}_{\tau_1}}\big\|\mcal{F}^{-1}g_{2}(x,t)\big\|_{L^{4}_{t}L^{2}_{x}}\big\|\mcal{F}^{-1}g_3(x,t)\big\|_{L^{4}_{t}L^{2}_{x}} \\
				&\ls & \|g_1\|_{L^{2}_{k_1}L^{2}_{\tau_1}}\|\mcal{F}^{-1}g_2\|_{X^{\a_2}_{0,\frac14,\lam}}\|\mcal{F}^{-1}g_3\|_{X^{\a_1}_{0,\frac14,\lam}}\\
				&\ls & \prod_{i=1}^{3}\|f_{i}\|_{L^{2}(\m{Z}_{\lam}\times\m{R})},
			\end{eqnarray*}
			where the last inequality is due to (\ref{res, g}).
			
			\item {\bf Case 2.2.2.} $\text{MAX}=\la L_{2}\ra$ or  $\la L_{3}\ra$.  The argument is similar to that in Case 2.2.1,  therefore omitted.
		\end{itemize}
		
		{\bf Proof of (\ref{nd 2-d, l2-aux})}. We basically reduce the proof of (\ref{nd 2-d, l2-aux}) to the following two cases. In different regions, we have the flexibility to choose which case to estimate.
		
		\begin{itemize}
			\item First, by Cauchy-Schwartz inequality and duality, the proof of (\ref{nd 2-d, l2-aux}) reduces to establish 
			\be\label{nd 2-d, aux, reduce 1}
			\i_{A_{\lam}(k_{3},\tau_{3})} \frac{|k_{2}|\la k_{3}\ra^{s}}{\la k_1\ra^{s}\la k_2\ra^{s}}\frac{\prod_{i=1}^{3}|f_{i}(k_i,\tau_i)|}{\la L_1\ra^{\frac12}\la L_2\ra^{\frac12}\la L_3\ra^{\frac12-}}\leq C\lam^{p}\prod_{i=1}^{3}\|f_{i}\|_{L^{2}(\m{Z}_{\lam}\times\m{R})}.\ee
			
			\item Secondly, for fixed $k_3$, if $L_{3}$ is restricted to some set $\O(k_{3})$, then it follows from Holder's inequality in $\tau_3$ that
			\[\text{LHS of (\ref{nd 2-d, l2-aux})}\ls 
			\Bigg\|\bigg(\i_{\m{R}}\frac{\mb{1}_{\{L_{3}\in\O(k_{3})\}}}{\la L_{3}\ra}\,d\tau_{3}\bigg)^{\frac{1}{2}}\bigg\|\i_{A_{\lam}(k_3,\tau_3)}\frac{|k_{2}|\la k_{3}\ra^{s}}{\la k_1\ra^{s}\la k_2\ra^{s}}\frac{|f_{1}f_{2}|}{\la L_{1}\ra^{\frac12}\la L_{2}\ra^{\frac12}\la L_{3}\ra^{\frac12}}\bigg\|_{L^{2}_{\tau_{3}}}\Bigg\|_{L^{2}_{k_{3}}}.\]
			So if we can prove 
			\be\label{small int of L_3}
			\i_{\m{R}}\frac{\mb{1}_{\{L_{3}\in\O(k_{3})\}}}{\la L_{3}\ra}\,d\tau_{3}\ls \lam^{p},\ee
			then (\ref{nd 2-d, l2-aux}) will follow from (\ref{small int of L_3}), duality and (\ref{nd 2-d, l2-reg}).
		\end{itemize}
		
		If $\la L_1\ra=\text{MAX}$ or $\la L_2\ra=\text{MAX}$ or $\la L_3\ra=\text{MAX}$  with $\la L_{1}\ra\la L_{2}\ra\gs \la L_{3}\ra^{\frac{1}{100}}$, then we can apply the similar argument as that in the proof of 
		(\ref{nd 2-d, l2-reg}) to justify (\ref{nd 2-d, aux, reduce 1}). 
		
		If $\la L_3\ra=\text{MAX}$  with $\la L_{1}\ra\la L_{2}\ra\ll \la L_{3}\ra^{\frac{1}{100}}$, then we will  justify (\ref{small int of L_3}).  Similar to the proof of (\ref{nd 2-d, l2-reg}), the most delicate situation is when $\frac{k_1}{k_2}$ is very near the root $x_{1r}$ or $x_{2r}$.  Without loss of generality, we will only investigate the region where $\big|\frac{k_1}{k_2}-x_{1r}\big|< \frac{d_{r}}{\lam |k_2|}$ in the rest argument (see the above Case 2.2). 
		
		Firstly, it follows from $\big|\frac{k_1}{k_2}-x_{1r}\big|< \frac{d_{r}}{\lam |k_2|}$ that 
		\[\Big|k_1+\frac{x_{1r}k_3}{1+x_{1r}}\Big|<\frac{d_r}{(1+x_{1r})\lam}.\]
		So when $k_3$ is fixed, the choice of $k_1$ is at most $\frac{2d_r}{1+x_{1r}}$. Secondly, 
		\[\la L_{3}+H_{2}\ra=\big\la L_{3}-\sum_{i=1}^{3}L_i\big\ra=\la L_1+L_2\ra\ll \la L_3\ra^{\frac{1}{100}}.\]
		As a result, $|H_2|\sim |L_3|\gg 1$ and $\la L_{3}+H_2\ra\ll |H_2|^{\frac{1}{100}}$. So there exists a small constant $\delta\in \big(0,\frac{1}{10}\big)$ such that $\la L_{3}+H_2\ra\leq \delta\,|H_2|^{\frac{1}{100}}$. Fix this $\delta$.  For any $k_3\in\m{Z}_{\lam}$, define
		\[\begin{split}
			\O^{\delta}(k_{3})&=\Big\{\eta\in\m{R}:  \exists\, k_{1},k_{2}\in\m{Z}_{\lam} \;\text{such that}\; \sum_{i=1}^{3}k_{i}=0, \Big|k_1+\frac{x_{1r}k_3}{1+x_{1r}}\Big|<\frac{d_r}{(1+x_{1r})\lam}\\ 
			&\qquad \text{and\;}\big\la\eta + H_2(k_1,k_2,k_3)\big\ra\leq \delta |H_2(k_1,k_2,k_3)|^{\frac{1}{100}}\Big\}.
		\end{split}\]
		Then $\tau_{3}-\phi^{\a_1}(k_3)=L_3\in\O^{\delta}(k_3)$. For any $M\geq 1$, define
		$\O^{\delta}_{M}(k_3)=\{\eta\in\O^{\delta}(k_3): M/2\leq |\eta|\leq 2M\}$.
		For any $\eta\in \O^{\delta}_{M}(k_3)$, the choices of $k_1$, as mentioned above, is at most $\frac{2d_r}{1+x_{1r}}$. On the other hand, it follows from 
		\[\big\la\eta + H_2(k_1,k_2,k_3)\big\ra\leq \delta |H_2(k_1,k_2,k_3)|^{\frac{1}{100}}\]
		that $|H_2|\sim \eta\sim M$. So $\big\la\eta + H_2(k_1,k_2,k_3)\big\ra\leq \delta M^{\frac{1}{100}}$. As a result, $\big|\O_{M}^{\delta}(k_3)\big|\ls M^{\frac{1}{100}}$. By the change of variable $\eta=\tau_{3}-\phi_{3}(k_{3})=L_{3}$,
		\begin{eqnarray*}
			\i_{\m{R}}\frac{\mb{1}_{\{L_{3}\in\O^{\delta}(k_{3})\}}}{\la L_{3}\ra}\,d\tau_{3}=\i_{\O^{\delta}(k_{3})}\frac{d\eta}{\la\eta\ra}&=&\int_{|\eta|\leq \lam^{6}}\frac{d\eta}{\la\eta\ra}+\sum_{M:\, \text{dyadic},\,M>\lam^{6}}\int_{\O^{\delta}_{M}(k_{3})}\frac{d\eta}{\la\eta\ra} \\
			&\ls &\ln(1+\lam^{6})+\sum_{M:\, \text{dyadic},\,M>\lam^{6}}\frac{|\O^{\delta}_{M}(k_{3})|}{M}\\
			&\ls & \ln(1+\lam^{6})+\sum_{M:\, \text{dyadic},\,M>\lam^{6}}M^{-\frac{1}{2}}\ls \lam^{p},
		\end{eqnarray*}
		which verifies (\ref{small int of L_3}).
	\end{proof}

	\section{Sharpness of the bilinear estimates}
	\label{Sec, proof of sharp bilin}
	
	In this section we will prove Propositions \ref{Prop, sharp bilin est, nondiv 1} and \ref{Prop, sharp bilin est, nondiv 2} which justify the sharpness of the bilinear estimates in Lemmas \ref{Lemma, bilin est, nondiv 1} and \ref{Lemma, bilin est, nondiv 2}. Since their proofs are similar, we will only prove Proposition \ref{Prop, sharp bilin est, nondiv 2}. Throughout this section, 
	\[A:=\Big\{(\vec{k},\vec{\tau})\in \m{Z}^{3}\times\m{R}^{3}:\sum_{i=1}^{3}k_{i}=\sum_{i=1}^{3}\tau_{i}=0\Big\}.\]
	The following lemma will be frequently used in this section.
	
	\begin{lemma}\label{Lemma, area of convolution}
		Let $E_{i}\subseteq\m{Z}\times\m{R}\,(1\leq i\leq 3)$ be bounded regions such that $E_{1}+E_{2}\subseteq -E_{3}$,  i.e., 
		\be\label{comp region}
		-(k_{1}+k_{2}, \tau_{1}+\tau_{2})\in E_{3},\quad \forall\, (k_{i},\tau_{i})\in E_{i},\,1\leq i\leq 2.\ee
		Then 
		\[\int\limits_{A}\prod_{i=1}^{3}\mb{1}_{E_{i}}(k_{i},\tau_{i})= |E_{1}\|E_{2}|.\]
	\end{lemma}
	\begin{proof}
		Rewriting the left hand side as 
		\[\iint_{E_{1}}\bigg(\iint_{E_{2}}\mb{1}_{E_{3}}\big(-(k_{1}+k_{2}),-(\tau_{1}+\tau_{2})\big) \,dk_{2}\,d\tau_{2}\bigg)\,dk_{1}\,d\tau_{1},\]
		then the conclusion follows from (\ref{comp region}).
	\end{proof}

	\subsection{Proof of Proposition \ref{Prop, sharp bilin est, nondiv 2}}
	
	Similar to Lemma \ref{Lemma, bilin to weighted l2}, (\ref{bilin est, nondiv 2, reg}) is equivalent to 
	\be\label{sharp weighted l2, nondiv 2, reg}
	\i_{A}\frac{k_{2}\la k_{3}\ra^{s}}{\la k_{1}\ra^{s}\la k_{2}\ra^{s}\la L_{1}\ra^{b}\la L_{2}\ra^{b}\la L_{3}\ra^{1-b}}\prod_{i=1}^{3}f_{i}(k_{i},\tau_{i})\leq C\prod_{i=1}^{3}\|f_{i}\|_{L^{2}(\m{Z}\times\m{R})}, \quad \forall \{f_i\}_{i=1,2,3}\ee
	where 
	\[L_{1}=\tau_{1}-\a_{1}k_{1}^{3},\quad L_{2}=\tau_{2}-\a_{2}k_{2}^{3},\quad  L_{3}=\tau_{3}-\a_{1}k_{3}^{3}.\]
	The corresponding  resonance function $H_2$ is given by (\ref{H2}):
	\be\label{res fun,  nd 2}
	H_2(k_{1},k_{2},k_{3})=-3\a_1 k_2\Big[k_1^2+k_1k_2+\frac{1-r}{3}k_2^2\Big],\ee
	where $r=\frac{\a_2}{\a_1}$. If $k_2\neq 0$, then $H_{2}$ can be rewritten as in (\ref{H2 compact}):
	\be\label{res fun-compact, nd 2}
	H_2(k_{1},k_{2},k_{3})=-3\a_1 k_2^3 h_{r}\Big(\frac{k_1}{k_2}\Big),\ee
	where $h_{r}(x)=x^2+x+\frac{1-r}{3}$.

	\begin{proof}[\bf Proof of Part (a)]
		Suppose there exist $r<\frac14$, $s<-\frac14$ and $b\in\m{R}$ such that (\ref{sharp weighted l2, nondiv 2, reg}) holds.
		\begin{itemize}
			\item For large $N$, define $f_{i}=-\mb{1}_{B_{i}}$ for $1\leq i\leq 3$, where
			\begin{align*}
				B_{1}&=\{(k_{1},\tau_{1}): k_{1}= N,\quad  |\tau_{1}-\a_{1}N^{3}|\leq 1\},\\
				B_{2}&=\{(k_{2},\tau_{2}): k_{2}=-N,\quad |\tau_{2}+\a_{2}N^{3}|\leq 1\},\\
				B_{3}&=\{(k_3,\tau_3): k_{3}=0,\quad |\tau_{3}+(\a_{1}-\a_{2})N^{3}|\leq 2\}.
			\end{align*}
			Then $B_{1}+B_{2}\subseteq -B_{3}$. In addition,  for any $(k_i,\tau_i)\in B_i$, $1\leq i\leq 3$,  $\la L_{1}\ra\sim 1$,$\la L_{2}\ra\sim 1$ and it follows from (\ref{res fun-compact, nd 2}) that $ |H_2(k_1,k_2,k_3)|\sim N^{3}$. Therefore,  $\la L_{3}\ra\sim N^{3}$ and we conclude from (\ref{sharp weighted l2, nondiv 2, reg}) that
			\begin{align*}
				\frac{N}{N^{2s}N^{3(1-b)}}\i_{A}\prod_{i=1}^{3}\mb{1}_{B_{i}}(k_i,\tau_i) \ls C\prod_{i=1}^{3}|B_{i}|^{\frac12}.
			\end{align*}
			Noticing $|B_{i}|\sim 1$, applying Lemma \ref{Lemma, area of convolution} leads to  $N^{3b-2s-2}\ls C$. In other words, 
			\be\label{cond 1 on s, nd 2}
			3b-2s-2\leq 0.\ee
			
			\item Similarly, define $f_{i}=-\mb{1}_{B_{i}}$ for $1\leq i\leq 3$, where
			\begin{align*}
				B_{1}&=\{(k_{1},\tau_{1}): k_{1}= N,\quad  |\tau_{1}-\a_{1}N^{3}|\leq 1\},\\
				B_{3}&=\{(k_{3},\tau_{3}): k_{3}=0,\quad |\tau_{3}|\leq 1\},\\
				B_{2}&=\{(k_{2},\tau_{2}): k_{2}=-N,\quad |\tau_{2}+\a_{1}N^{3}|\leq 2\}.
			\end{align*}
			Then $B_{1}+B_{3}\subseteq -B_{2}$ and by similar argument, we conclude
			\be\label{cond 2 on s, nd 2}
			1-2s-3b\leq 0.\ee
		\end{itemize}
		(\ref{cond 1 on s, nd 2}) and (\ref{cond 2 on s, nd 2}) together yields $s\geq -\frac14$, which contradicts to the assumption $s<-\frac14$. In addition, when $s=-\frac14$, $b$ has to be exactly $\frac12$.
	\end{proof}
	
	\begin{proof}[\bf Proof of Part (b)]
		Under the additional assumption $\wh{w_1}(0,\cdot)=0$, (\ref{bilin est, nondiv 2, reg}) is equivalent to (\ref{sharp weighted l2, nondiv 2, reg}) with the additional restriction $f_{1}(0,\cdot)=0$. When $r=1$, writing $\a_1=\a_2=\a$, then $H_2(k_1,k_2,k_3)=3\a k_{1}k_{2}k_{3}$. Define 
		$f_{i}=-\mb{1}_{B_{i}}$ for $1\leq i\leq 3$, where
		\begin{align*}
			B_{1}&=\{(k_{1},\tau_{1}): k_{1}= N,\quad  |\tau_{1}-\a N^{3}|\leq 1\},\\
			B_{2}&=\{(k_{2},\tau_{2}): k_{2}=-N,\quad |\tau_{2}+\a N^{3}|\leq 1\},\\
			B_{3}&=\{(k_3,\tau_3): k_{3}=0,\quad |\tau_{3}|\leq 2\}.
		\end{align*}
		Then $B_{1}+B_{2}\subseteq -B_{3}$. In addition,  for any $(k_i,\tau_i)\in B_i$, $1\leq i\leq 3$,  $\la L_{1}\ra\sim 1$, $ \la L_{2}\ra\sim 1$ and $ H_2(k_1,k_2,k_3)=0$, which implies $\la L_{3}\ra\sim 1$. Then it follows from (\ref{sharp weighted l2, nondiv 2, reg}) that
		\begin{align*}
			\frac{N}{N^{2s}}\i_{A}\prod_{i=1}^{3}\mb{1}_{B_{i}}(k_i,\tau_i) \ls C\prod_{i=1}^{3}|B_{i}|^{\frac12}.
		\end{align*}
		Noticing $|B_{i}|\sim 1$, applying Lemma \ref{Lemma, area of convolution} leads to  $s\geq \frac12$.
	\end{proof}
	
	\begin{proof}[\bf Proof of Part (c)]
		Analogous to part (b) but without the restriction $\wh{w_1}(0,\cdot)$, so $f_1(0,\cdot)$ is not required to be zero anymore. So it is valid to define $f_{i}=-\mb{1}_{B_{i}}$, where 
		\begin{align*}
			B_{1}&=\{(k_{1},\tau_{1}): k_{1}= 0,\quad  |\tau_{1}|\leq 1\},\\
			B_{2}&=\{(k_{2},\tau_{2}): k_{2}=N,\quad |\tau_{2}-\a N^3|\leq 1\},\\
			B_{3}&=\{(k_3,\tau_3): k_{3}=-N,\quad |\tau_{3}+\a N^{3}|\leq 2\}.
		\end{align*}
		Then  $B_{1}+B_{2}\subseteq -B_{3}$. In addition, for any $(k_i,\tau_i)\in B_i$, $1\leq i\leq 3$,  $\la L_{1}\ra\sim 1$, $\la L_{2}\ra\sim 1$ and $H_2(k_1,k_2,k_3)=0$, which implies $\la L_{3}\ra\sim 1$. Then it follows from (\ref{sharp weighted l2, nondiv 2, reg}) that
		\begin{align*}
			N\i_{A}\prod_{i=1}^{3}\mb{1}_{B_{i}}(k_i,\tau_i)\ls C\prod_{i=1}^{3}|B_{i}|^{\frac12},
		\end{align*}
		which implies $N\ls C$ due to Lemma \ref{Lemma, area of convolution}. But this is impossible since $N$ can be arbitrarily large.
	\end{proof}

	\begin{proof}[\bf Proof of Part (d)]
		The following proof is in the similar spirit as that in (\cite{Oh09}, Proposition 3.9).  
		
		The proofs for the case $r\in[\frac14,1)$ and the case $r\in(1,\infty)$ are analogous, so we will just assume $r\in[\frac14,1)$.  Suppose there exist $s<s_{r}$ and $b\in\m{R}$ such that (\ref{bilin est, nondiv 2, reg}) holds. Then (\ref{sharp weighted l2, nondiv 2, reg}) holds for some constant $C=C(\a_1,\a_2,s,b)$. Since $r\in[\frac14,1)$, the function $h_{r}$ has two roots 
		\[x_{1r}=-\frac12-\frac{\sqrt{12r-3}}{6},\quad   x_{2r}=-\frac12+\frac{\sqrt{12r-3}}{6}.\]
		In addition, $-1<x_{1r}\leq x_{2r}<0$ and $\mu(x_{1r})=\mu(x_{2r})=\mu(\sqrt{12r-3})=\sigma_r$. When $k_2\neq 0$, the resonance function $H_2$ can be written as 
		\[H_{2}(k_1,k_2,k_3)=-3\a_1 k_2^3\Big(\frac{k_1}{k_2}-x_{1r}\Big)\Big(\frac{k_1}{k_2}-x_{2r}\Big).\]
		In the following, we will first show that $\frac13\leq b\leq \frac23$ and then derive a contradiction to $s<s_{r}$.
		
		For large positive $N$, define $f_{i}=\mb{1}_{B_{i}}$ for $1\leq i\leq 3$, where
		\begin{align*}
			B_{1}&=\{(k_{1},\tau_{1}): k_{1}= 0,\quad  |\tau_{1}|\leq 1\},\\
			B_{2}&=\{(k_{2},\tau_{2}): k_{2}=N,\quad |\tau_{2}-\a_{2}N^3|\leq 1\},\\
			B_{3}&=\{(k_3,\tau_3): k_{3}=-N,\quad |\tau_{3}+\a_2 N^{3}|\leq 2\}.
		\end{align*}
		Then $B_{1}+B_{2}\subseteq -B_{3}$. In addition, for any $(k_i,\tau_i)\in B_i$, $1\leq i\leq 3$, $\la L_{1}\ra\sim 1$,$\la L_{2}\ra\sim 1$ and 
		\[ |H_2(k_1,k_2,k_3)|=|\a_2 N^{3}-\a_1 N^{3}|\sim N^{3},\]
		which implies $\la L_{3}\ra\sim N^{3}$. Then it follows from (\ref{sharp weighted l2, nondiv 2, reg}) that
		\begin{align*}
			\frac{N^{1+s}}{N^{s}N^{3(1-b)}}\i_{A}\prod_{i=1}^{3}\mb{1}_{B_{i}}(k_i,\tau_i) \ls C\prod_{i=1}^{3}|B_{i}|^{\frac12}.
		\end{align*}
		Noticing $|B_{i}|\sim 1$, applying Lemma \ref{Lemma, area of convolution} leads to  $b\leq \frac23$. 
		On the other hand,  define $f_{i}=\mb{1}_{B_{i}}$ for $1\leq i\leq 3$, where
		\begin{align*}
			B_{1}&=\{(k_{1},\tau_{1}): k_{1}= 0,\quad  |\tau_{1}|\leq 1\},\\
			B_{3}&=\{(k_{3},\tau_{3}): k_{3}=-N,\quad |\tau_{3}+\a_{1}N^3|\leq 1\},\\
			B_{2}&=\{(k_2,\tau_2): k_{2}=N,\quad |\tau_{2}-\a_1 N^3|\leq 2\}.
		\end{align*}
		Then $B_{1}+B_{3}\subseteq -B_{2}$ and by applying similar arguments, we conclude that  $b\geq \frac13$. Thus, $\frac13\leq b\leq \frac23$. 
		
		Next, we discuss two situations depending on whether $\sqrt{12r-3}$ is a rational number or not.
		
		\begin{itemize}
			\item {\bf Case 1.} $\sqrt{12r-3}\in\m{Q}$.  
			
			In this case, $x_{1r}\in\m{Q}$, $\sigma_r=1$ and $s_r=1$. In addition, since $-1<x_{1r}<0$, we can assume $x_{1r}=\frac{p_1}{q_1}$ for some $p_1,\,q_1\in\m{Z}$ with $p_1<0$, $q_1>0$ and $p_{1}+q_{1}>0$.  For any large positive $N$, define $f_{i}=-\mb{1}_{B_{i}}$ for $1\leq i\leq 3$, where
			\begin{align*}
				B_{1}&=\{(k_{1},\tau_{1}): k_{1}= p_1 N,\quad  |\tau_{1}-\a_{1}(p_1 N)^{3}|\leq 1\},\\
				B_{2}&=\{(k_{2},\tau_{2}): k_{2}=q_1 N,\quad |\tau_{2}-\a_{2}(q_1 N)^3|\leq 1\},\\
				B_{3}&=\big\{(k_3,\tau_3): k_{3}=-(p_1+q_1)N,\quad \big|\tau_{3}+\a_{1}(p_1 N)^{3}+\a_2(q_1 N)^3\big|\leq 2\big\}.
			\end{align*}
			Then $B_{1}+B_{2}\subseteq -B_{3}$. In addition, for any $(k_i,\tau_i)\in B_i$, $1\leq i\leq 3$, $\la L_1\ra\sim 1$, $\la L_2\ra\sim 1$ and $\frac{k_1}{k_2}=\frac{p_1}{q_1}=x_{1r}$. As a result, 
			\[H_2(k_1,k_2,k_3)=-3\a_1 k_2^3\Big(\frac{k_1}{k_2}-x_{1r}\Big)\Big(\frac{k_1}{k_2}-x_{2r}\Big)=0,\]
			which implies $\la L_3\ra\sim 1$. Then it follows from (\ref{sharp weighted l2, nondiv 2, reg}) that
			\begin{align*}
				\frac{N^{1+s}}{N^{2s}}\i_{A}\prod_{i=1}^{3}\mb{1}_{B_{i}}(k_i,\tau_i) \ls C\prod_{i=1}^{3}|B_{i}|^{\frac12}.
			\end{align*}
			Noticing $|B_{i}|\sim 1$, applying Lemma \ref{Lemma, area of convolution} leads to  $s\geq 1$, which contradicts to $s<s_r=1$.

			\item  {\bf Case 2.} $\sqrt{12r-3}\notin\m{Q}$.
			
			Since $s<s_r$ by the assumption, there exists $\eps_0>0$ such that $s+\eps_0<s_r$.  Recalling $\mu(x_{1r})=\sigma_r$, so $x_{1r}$ is approximable with power $\sigma_r-\eps_0$.  Hence, it follows from Lemma \ref{Lemma, seq for IE} that there exists a sequence $\{(m_j,\,n_j)\}_{j=1}^{\infty}\subset\m{Z}\times\m{Z}^*$ such that $n_j>0$,  $\lim\limits_{j\rightarrow\infty}n_j=\infty$ and 
			\[0<\Big|\frac{m_j}{n_j}-x_{1r}\Big|< \frac{1}{n_j^{\sigma_r-\eps_0}}.\]
			When $j$ is large enough, $\frac{m_j}{n_j}$ will be very close to $x_{1r}$. Since $-1<x_{1r}<0$, it implies
			$m_j<0$, $m_j+n_j>0$ and $|m_j|\sim m_j+n_j\sim n_j$. In addition, $|\frac{m_j}{n_j}-x_{2r}|\approx |x_{1r}-x_{2r}|$ which is a positive constant only depending on $r$. As a result, 
			\[\big|H_2\big(m_j,n_j,-(m_j+n_j)\big)\big| = \Big|3\a_1 n_j^3\Big(\frac{m_j}{n_j}-x_{1r}\Big)\Big(\frac{m_j}{n_j}-x_{2r}\Big)\Big| \ls n_j^{3-\sigma_r+\eps_0}.\]
			Denote 
			\be\label{power bdd}
			\zeta_r=\max\{\eps_0,3-\sigma_r+\eps_0\}.\ee
			Then no matter $\sigma_r<3$ or $\sigma_r\geq 3$, it always holds that
			\be\label{upper bdd for H, nd 2}
			\big\la H_2\big(m_j,n_j,-(m_j+n_j)\big)\big\ra =1+ \big|H_2\big(m_j,n_j,-(m_j+n_j)\big)\big| \ls n_j^{\zeta_r}.\ee
			
			\begin{itemize}
				\item For any large $j$ as in the above discussion, define $f_{i}=\mb{1}_{B_{i}}$ for $1\leq i\leq 3$, where
				\begin{align*}
					B_{1}&=\{(k_{1},\tau_{1}): k_{1}= m_j,\quad  |\tau_{1}-\a_{1}m_j^{3}|\leq 1\},\\
					B_{2}&=\{(k_{2},\tau_{2}): k_{2}=n_j,\quad |\tau_{2}-\a_{2}n_j^3|\leq 1\},\\
					B_{3}&=\big\{(k_3,\tau_3): k_{3}=-(m_j+n_j),\quad \big|\tau_{3}+\a_{1}m_j^{3}+\a_{2}n_j^3\big|\leq 2\big\}.
				\end{align*}
				Then $B_{1}+B_{2}\subseteq -B_{3}$. In addition, for any $(k_i,\tau_i)\in B_i$, $1\leq i\leq 3$, $\la L_1\ra\sim 1$, $\la L_2\ra\sim 1$ and 
				\[\la H_2(k_1,k_2,k_3)\ra= \big\la H_2\big(m_j,n_j,-(m_j+n_j)\big)\big\ra\ls n_j^{\zeta_r},\]
				which implies $\la L_3\ra\ls n_j^{\zeta_r}$. Then it follows from (\ref{sharp weighted l2, nondiv 2, reg})  and $|m_j|\sim m_j+n_j\sim n_j$ that
				\begin{align*}
					\frac{n_j^{1+s}}{n_j^{2s}n_j^{\zeta_r(1-b)}}\i_{A}\prod_{i=1}^{3}\mb{1}_{B_{i}}(k_i,\tau_i) \ls C\prod_{i=1}^{3}|B_{i}|^{\frac12}.
				\end{align*}
				Noticing $|B_{i}|\sim 1$, applying Lemma \ref{Lemma, area of convolution} leads to 
				\be\label{sharp nd2, res, case 2, cond 1 on s}
				1-s-\zeta_r(1-b)\leq 0.\ee
				
				\item On the other hand, define $f_{i}=\mb{1}_{B_{i}}$ for $1\leq i\leq 3$, where
				\begin{align*}
					B_{1}&=\{(k_{1},\tau_{1}): k_{1}= m_j,\quad  |\tau_{1}-\a_{1}m_j^{3}|\leq 1\},\\
					B_{3}&=\big\{(k_3,\tau_3): k_{3}=-(m_j+n_j),\quad \big|\tau_{3}+\a_{1}(m_j+n_j)^{3}\big|\leq 1\big\},\\
					B_{2}&=\{(k_{2},\tau_{2}): k_{2}=n_j,\quad |\tau_{2}+\a_{1}m_j^{3}-\a_{1}(m_j+n_j)^{3}|\leq 2\}.
				\end{align*}
				Then $B_{1}+B_{3}\subseteq -B_{2}$. In addition,  by applying similar arguments, we conclude that 
				\be\label{sharp nd2, res, case 2, cond 2 on s}
				1-s-\zeta_r b\leq 0.\ee
			\end{itemize}
			Adding (\ref{sharp nd2, res, case 2, cond 1 on s}) and (\ref{sharp nd2, res, case 2, cond 2 on s}) together yields $s\geq 1-\frac{\zeta_r}{2}$. Recalling $s+\eps_0<s_r$, so
			\be\label{index relation}
			1-\frac{\zeta_r}{2}\leq s<s_r-\eps_0.\ee
			If $\sigma_r\geq 3$, then $s_r=1$ and $\zeta_r=\eps_0$, which contradicts to (\ref{index relation}). If $2\leq\sigma_r< 3$, then $s_r=\frac{\sigma_r-1}{2}$ and $\zeta_r=3-\sigma_r+\eps_0$, which also contradicts to (\ref{index relation}).
		\end{itemize}
	\end{proof}
	
	\appendix
	\section{Well-posedness for the Hirota-Satsuma systems (\ref{H-S system})}
	\label{Appendix,  H-S}
	
	In this appendix, we summarize the analytical well-posedness results on the Hirota-Satsuma systems (\ref{H-S system}).
	
	\begin{theorem}\label{Thm, LWP for H-S-1}
		The Hirota-Satsuma system (\ref{H-S system}) is A-LWP in  $\mcal{H}_{2}^s$ if one of the following conditions is satisfied.
		\begin{enumerate}[(1)]
			\item $a_1\in(-\infty,\frac14)\setminus\{0\}$ and $s\geq -\frac14$;
			
			\item $a_{1}=1$, $c_{12}=0$ and $s\geq \frac12$;
			
			\item $a_1\in[\frac14,\infty)\setminus\{1\}$, and $s\geq 1$ or $s>s_{a_1}$ (equivalently $s\geq \min\{1,\,s_{a_1}+\}$).
		\end{enumerate}
	\end{theorem}
	
	Due to the following conserved energies for (\ref{H-S system}), the A-GWP follows directly from Theorem \ref{Thm, LWP for H-S-1}.
	\be\label{H-S energy}\begin{split}
		E_{1}(u,v) &= \int u^{2}+\frac{c_{12}}{3}v^{2}\,dx,\\
		E_{2}(u,v) &= \int (1-a_1)u_{x}^{2}+c_{12}v_{x}^{2}-2(1-a_{1})u^{3}-c_{12}uv^{2}\,dx.
	\end{split}\ee
	
	\begin{theorem}\label{Thm, GWP for H-S-1}
		Let $c_{12}>0$. Then the Hirota-Satsuma system (\ref{H-S system}) is  A-GWP  in  $\mcal{H}^{s}_{2}$ if one of the following conditions is satisfied.
		\begin{enumerate}[(1)]
			\item $a_1\in(-\infty,\frac14)\setminus\{0\}$ and $s\geq 0$;
			
			\item $a_1\in[\frac14,1)$ and 
			$s\geq 1$.
		\end{enumerate}
	\end{theorem}
	

	\bigskip
	
	\thanks{(X. Yang) Department of Mathematics, University of California, Riverside, CA 92521, USA.}
	
	\thanks{Email: xiny@ucr.edu}
	
	\bigskip
	
	\thanks{(B.-Y. Zhang) Department of Mathematical Sciences, University of Cincinnati, Cincinnati, OH 45221, USA.}
	
	\thanks{Email: zhangb@ucmail.uc.edu}
	
\end{document}